\documentclass[10pt]{article}
\usepackage{amsmath,amsfonts,amssymb,amsthm,mathrsfs}
\usepackage{dsfont}
\usepackage[a4paper,vmargin={3.5cm,3.5cm},hmargin={2.5cm,2.5cm}]{geometry}
\usepackage[font=sf, labelfont={sf,bf}, margin=1cm]{caption}
\usepackage{graphicx,graphics}
\usepackage{float}
\usepackage{enumerate,enumitem}
\usepackage{nicefrac}
\usepackage{epsfig}
\usepackage{latexsym}
\usepackage[utf8]{inputenc}
\usepackage{ae,aecompl}
\usepackage[english]{babel}
 \usepackage[colorlinks=true]{hyperref}
\usepackage{pstricks}

\usepackage{cleveref}
  \crefname{theorem}{Theorem}{Theorems}
  \crefname{lemma}{Lemma}{Lemmas}
  \crefname{remark}{Remark}{Remarks}
  \crefname{proposition}{Proposition}{Propositions}
  \crefname{definition}{Definition}{Definitions}
  \crefname{corollary}{Corollary}{Corollaries}
  \crefname{section}{Section}{Sections}
  \crefname{figure}{Figure}{Figures}

\newcommand{\Z}{\mathbb{Z}}
\renewcommand{\P}{\mathbb{P}}
\newcommand{\Ind}{\mathbf{1}}

\newtheorem{theorem}[]{Theorem}
\newtheorem{proposition}[theorem]{Proposition}
\newtheorem{lemma}[theorem]{Lemma}
\newtheorem{definition}[]{Definition}

\newtheorem{citetheorem}[]{Theorem}

\newtheorem*{proposition*}{Proposition}

\theoremstyle{definition}
\newtheorem*{remark}{Remark}

\def\N{\ensuremath{\mathbb{N}}}

\def\R{\ensuremath{\mathbb{R}}}
\def\Z{\ensuremath{\mathbb{Z}}}

\def\ep{\varepsilon}

\def\F{\ensuremath{\mathcal{F}}}

\def\tif{\ensuremath{\text{ if }}}

\newcommand{\U}{\mathcal U}

\def\build#1_#2^#3{\mathrel{
\mathop{\kern 0pt#1}\limits_{#2}^{#3}}}

\usepackage[nottoc,numbib]{tocbibind}
\usepackage{mathsymb-linxiao}

\newcommand*{\Op}{\mathcal{O}_p}

\newcommand*{\g}{\mathbf{g}}
\newcommand*{\tg}{\tilde{g}}
\DeclareMathOperator{\arccot}{\operatorname{arccot}}

\newcommand{\mujs}{\mu_{\mathsf{JS}}}
\newcommand{\boldX}{\boldsymbol{\mathcal{X}}\0p}
\newcommand{\Dzeta}{(\boldsymbol{\Delta\zeta})^{\downarrow}}
\newcommand{\ql}{(\mathfrak{q}, \boldsymbol{\ell})}

\title{\bf \textsc{The perimeter cascade in critical Boltzmann quadrangulations decorated by an $O(n)$ loop model}}
\author{Linxiao Chen, Nicolas Curien, Pascal Maillard}
\date{}

\newcommand{\Addresses}{{
  \bigskip\footnotesize

\noindent  \textsc{University of Helsinki, Department of Mathematics and Statistics, P.O. Finland} ~and~ \textsc{Laboratoire de Math\'ematiques d'Orsay, Univ.~Paris--Sud, CNRS, Universit\'e Paris--Saclay, 91405 Orsay, France} ~and~ \textsc{Institut de Physique Th\'eorique, Universit\'e Paris-Saclay, CEA, CNRS, F-91191 Gif-sur-Yvette.}\par\nopagebreak
  \textit{E-mail address}: \texttt{linxiao.chen@helsinki.fi}	\medskip

\noindent\textsc{Laboratoire de Math\'ematiques d'Orsay, Univ.~Paris--Sud, CNRS, Universit\'e Paris--Saclay, 91405 Orsay, France} ~and~ \textsc{Institut Universitaire de France.}\par\nopagebreak
  \textit{E-mail address}: \texttt{nicolas.curien@gmail.com}	\medskip

\noindent\textsc{Laboratoire de Math\'ematiques d'Orsay, Univ.~Paris--Sud, CNRS, Universit\'e Paris--Saclay, 91405 Orsay, France.}\par\nopagebreak
  \textit{E-mail address}: \texttt{pascal.maillard@u-psud.fr}
}}

\begin{document}
\maketitle
\begin{abstract}
We study the branching tree of the perimeters of the nested loops in critical $O(n)$  model for $n \in (0,2)$ on random quadrangulations. We prove that after renormalization it converges towards an explicit continuous multiplicative cascade whose offspring distribution $(x_i)_{i \geq 1}$ is related to the jumps of a spectrally positive $\alpha$-stable L\'evy process with $\alpha= \frac{3}{2} \pm \frac{1}{\pi} \arccos(n/2)$ and for which we have the surprisingly simple and explicit transform
$$  \exptm{ \sum_{i \geq 1}(x_i)^\theta } = \frac{\sin(\pi (2-\alpha))}{\sin (\pi (\theta - \alpha))} \quad \mbox{for }\theta \in (\alpha, \alpha+1).$$
An important ingredient in the proof is a new formula of independent interest on first moments of additive functionals of the jumps of a left-continuous random walk stopped at a hitting time.
We also identify the scaling limit of the volume of the critical $O(n)$-decorated quadrangulation using the Malthusian martingale associated to the continuous multiplicative cascade.
\end{abstract}

\vspace{-1em}
\begin{figure}[!h]
\begin{center}
\includegraphics[width=14cm,page=1]{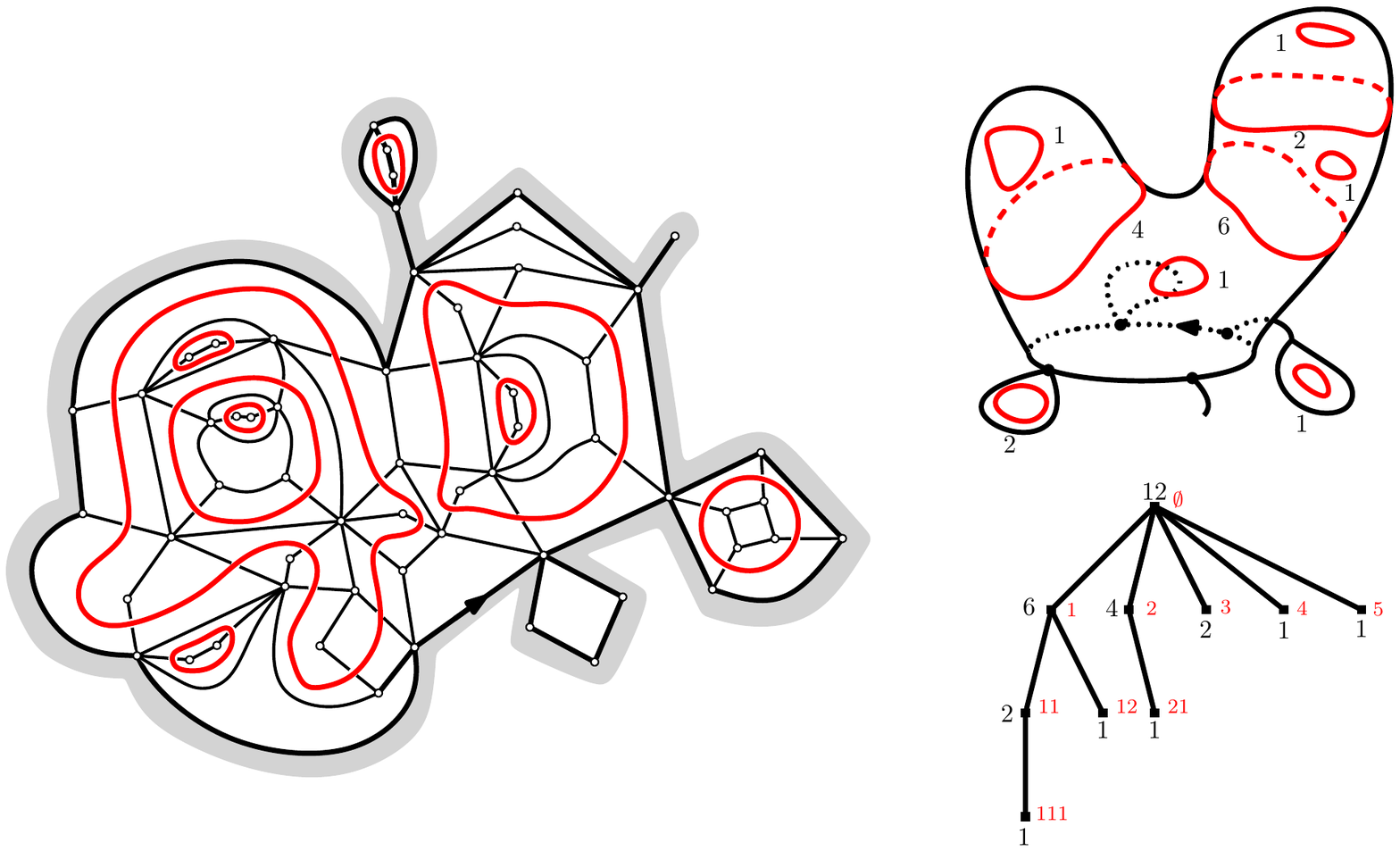}
\caption{ \label{fig:1}Left: a loop-decorated quadrangulation with a boundary of perimeter $24$. Top right: the topological representation of the quadrangulation as a ``cactus'' which highlights the nesting of the loops. Bottom right: the nesting tree labeled by the half-perimeters of the loops.}
\end{center}
\end{figure}
 
\clearpage

\section{Introduction}
We build on the work by  Borot, Bouttier \& Guitter \cite{BBG11} on critical Boltzmann quadrangulations decorated by an $O(n)$ loop model. Among other things, they showed that the so-called gasket associated to a critical $O(n)$-decorated random Boltzmann quadrangulation is for a certain choice of parameters a ``non-generic critical'' Boltzmann map, in the sense that the face weights have a polynomial decay $k^{-a}$ with $a=2 \pm \frac{1}{\pi} \arccos(n/2)$. In this work we analyze in detail the nested sequence of the perimeters of the loops  and show that it converges towards an explicit multiplicative cascade related to a stable L\'evy process with index $\alpha = a - \frac{1}{2}$. To properly state our results, let us first recall the setup of \cite{BBG11}.
\paragraph{Definition of the loop model.} In the terminology of \cite{BBG11} we work with the rigid $O(n)$ loop model on quadrangulations. Recall that in a rooted planar map $ \mathfrak{m}$, the face $ f_{ \mathrm{r}}$ to the right of the root edge is called the \emph{root face} or the  \emph{external face} (the other faces being internal faces). Its degree is called the \emph{perimeter} of the map. A \emph{quadrangulation with a boundary} is a rooted planar map $\mathfrak{q}$ whose internal faces all have degree four. A  \emph{(rigid) loop configuration} on a quadrangulation with a boundary is a set $\boldsymbol{\ell} = \{\ell_{1}, \ell_{2}, ...\}$ of disjoint undirected simple closed paths in the dual map which do not visit the external face, and with the additional constraint that when a loop visits a face of $ \mathfrak{q}$ it must cross it through opposite edges. In other words, the internal faces of $ \mathfrak{q}$ can only be of the following two types $$ \includegraphics[height=1cm]{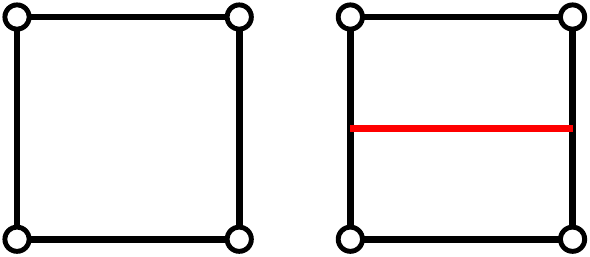},$$ see Fig.\ \ref{fig:1}. The pair $(\mathfrak{q},\boldsymbol{\ell})$ will henceforth be called a \emph{loop-decorated quandrangulation with a boundary}.

Given $n \in (0,2)$, $g \geq 0$ and $h \geq 0$, we define a measure $ \mathsf{w}$ on the set of all loop-decorated quadrangulations with a boundary by putting
\begin{eqnarray} \label{eq:weights}  \mathsf{w}_{(n;g,h)}(\ql) = g^{|\mathfrak{q}|-|\boldsymbol{\ell}|} h^{| \boldsymbol{\ell}|} n^{\# \boldsymbol{\ell}}  \end{eqnarray} where $| \mathfrak{q}|$ is the number of inner faces of $ \mathfrak{q}$,  $| \boldsymbol{\ell}|$ is the total length of the loops of $  \boldsymbol{\ell}$ and $\# \boldsymbol{\ell}$ is the number of loops in $ \boldsymbol{\ell}$. For example, the weight of the quadrangulation presented in Fig.\,\ref{fig:1} is $g^8h^{38}n^9$. We denote by $ \Op$ the set of all loop-decorated quadrangulations with a boundary whose external face has degree $2p$ and put 
 \begin{eqnarray}
\label{eq:defFp}	F_p(n;g,h) = \sum_{\ql\in\Op}  \mathsf{w}_{(n;g,h)}(\ql).
 \end{eqnarray}

If $F_p(n;g,h)$ is finite (it is not hard to see that the finiteness does not depend on the value of $p \geq 1$) the set of parameters $(n;g,h)$ is said to be admissible and we can define the normalized probability distribution on  $\Op$:
$$
	\prob\0p_{n;g,h}(\cdot) = \frac{ \mathsf{w}_{(n;g,h)}(\cdot)}{F_p(n;g,h)}.
$$
For large $p$, the geometry of a random map distributed according to $\proba\0p_{n;g,h}$ depends heavily on the parameters $(n;g,h)$. Borot, Bouttier \& Guitter \cite{BBG11} have classified this set of parameters into three categories called \emph{subcritical}, \emph{generic critical} or \emph{non-generic critical}, see Fig. \ref{fig:phase diagram}.

\begin{figure}
\begin{center}
\includegraphics[scale=0.9,page=1]{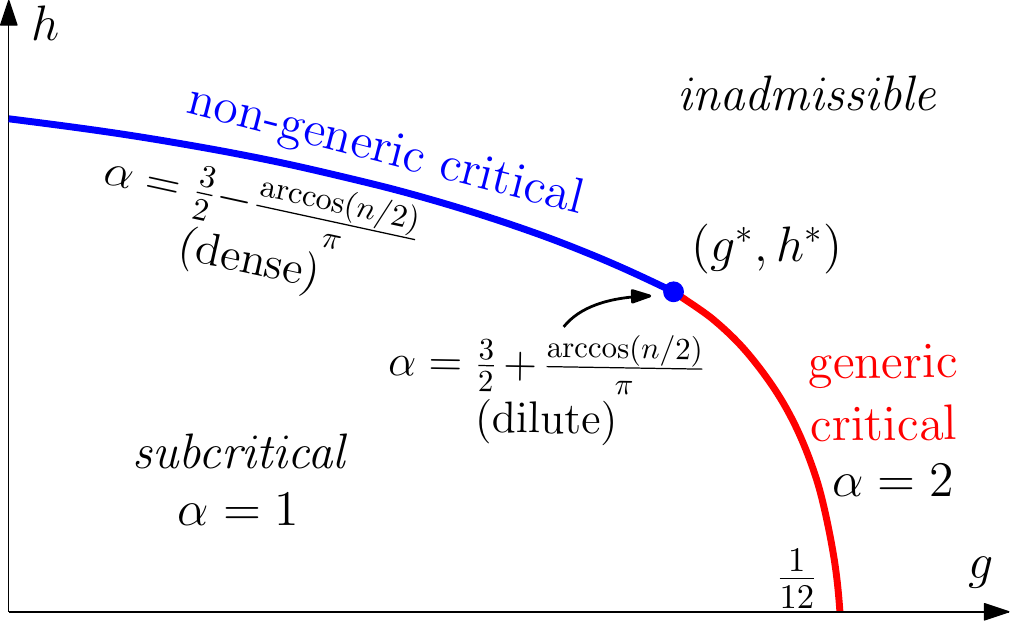}
\end{center}
\vspace{-1em}
\caption{Phase diagram of the loop-decorated quadrangulation model \cite{BBG11,BudOn} for a given value of $n \in (0,2)$ (the diagram is qualitatively the same for all such $n$).
For every admissible set of parameters $(n;g,h)$, the asymptotic form \eqref{eq:F_p asymptotic} holds. 
The different phases are characterized by different values of the exponent $\alpha$.
The critical line separates the sub-critical region (below) and the inadmissible region (above). 
}
\label{fig:phase diagram}
\end{figure}

In the subcritical case, roughly speaking the random maps distributed according to $\proba\0p_{n;g,h}$ are tree-like for large $p$ and are expected to converge in the scaling limit towards Aldous' CRT.
In the generic critical case, they are believed to behave as standard quadrangulations with a boundary and should converge towards the Brownian disk \cite{BM15}.
In the non-generic critical case however, the geometry of these maps remains elusive and the only available information we have is on their gasket \cite{LGM09}, see Section \ref{sec:gasket} for the definition.
 In particular, in this regime we have the asymptotic
\begin{equation} \label{eq:F_p asymptotic}
F_p(n;g,h) \underset{p\to\infty}{\sim} C \,\kappa^p \,p^{-\alpha -1/2},
\vspace{-1ex}
\end{equation}
for some $C>0, \kappa >0$ and where the exponent $\alpha$ satisfies 
\begin{equation}
\alpha =  \frac{3}{2} \pm \frac{1}{\pi} \arccos(n/2) \quad \in (1,2) \backslash \{ 3/2\}.
\label{eq:alpha_n}
\end{equation}
Here, the sign depends on the parameters $g$ and $h$, see Fig.~\ref{fig:phase diagram}.
The case $\alpha = \frac32-\frac1\pi \arccos(n/2) \in (1,\frac32)$  is called the dense case because in a suitable scaling limit, the loops are believed to touch themselves and each other, whereas in the dilute case  $\alpha = \frac32+ \frac1\pi \arccos(n/2) \in (\frac32,2)$ they are believed to be simple and not to touch each other.

To be precise, the work \cite{BBG11} must be completed by \cite[Appendix]{BudOn} in order to have a fully rigorous proof of the above phase diagram and in particular of the existence of non-generic critical parameter $(n;g,h)$. 

\begin{center}
\begin{minipage}{14cm}
\textit{In the rest of the paper we assume we are given a non-generic critical set of parameters $(n;g,h)$ with $n \in (0,2)$ and $h ,g \geq 0$  in the sense of Definition \ref{def:non-generic} below---in particular  \eqref{eq:F_p asymptotic} holds for a value of $\alpha \in  (1,3/2) \cup (3/2,2)$ that is fixed in the rest of the paper. We shall sometimes drop the implicit dependence in $(n;g,h)$ in what follows.}
\end{minipage}
\end{center}

\paragraph{Discrete and continuous cascades.}  We are interested in the perimeters and the nesting structure of the loops in a random loop-decorated quadrangulation distributed according to $ \prob\0p_{(n;g,h)}$ as $p\to\infty$. If $\ql$ is a loop-decorated quadrangulation with a boundary of perimeter $2p$, we can associate with it a random labeled tree as follows. We start with the so-called Ulam tree
$$\U = \bigcup_{n\ge 0} (\natural^*)^n.$$
Here and throughout we use French notation, i.e.~$\N = \{0,1,2,\ldots\}$ and $\N^* = \N\backslash\{0\}$ and set $(\natural^*)^0 = \{\varnothing\}$. If $u,v \in \U$ we write $uv$ for the concatenation of $u$ and $v$, and we write $\abs{u}=k$ if $u$ is s vertex in the $k$-th generation, i.e.\ $u\in(\natural^*)^k$.
Then we assign each loop $ \ell \in \boldsymbol{\ell}$ to a vertex of $ \U$ in the following fashion: First, the root vertex $\varnothing$ of $ \U$ is associated with an imaginary loop of length $2p$ surrounding the boundary of $  \mathfrak{q}$. Next we assign to the children $1,2,3,\cdots$ of $ \varnothing$ the outer-most loops of $\ql$---i.e.\,those loops which can be reached from the boundary of $ \mathfrak{q}$ without crossing any other loop---ranked by decreasing perimeter (if there is a tie we break it using a deterministic rule). We then continue genealogically inside each of these loops in the most obvious way, see Fig.\,\ref{fig:1}.
Although $\U$ is an infinite (even non locally finite) tree, the set of vertices attached to a loop of $ \boldsymbol{\ell}$ is a finite subtree of $ \U$. Once this is done, we define the labeling $$ \chi : u \in \U  \mapsto \chi(u) \in \natural$$ which is the half-perimeter of the loop associated to $u$ in $\ql$ or $0$ if there are no such loop. 

We now introduce the limiting continuous multiplicative cascade. Given a distribution $ \nu$ on $ (\real_+)^{\natural^*}$, let $\{(\xi_i\0u)_{i \geq 1} : u \in \U\}$ be an i.i.d.~family of (infinite) random vectors of law $\nu$. The multiplicative cascade with offspring distribution $\nu$ is then the random process $(Z(u))_{u\in\U}$ indexed by the Ulam tree such that $Z( \varnothing) =1$ and such that for any $u \in \U$ and any $i \in \natural^*$ we have $Z(ui) = Z(u) \cdot \xi_i ^{(u)}$. We will apply this to a particular law $\nu$: 
Let $(\zeta_t)_{t\ge0}$ be an $\alpha$-stable L\'evy process with no negative jumps started at 0; in other words for some constant $C>0$ we have $ \expt [\exp( -\lambda \zeta_t)] = \exp(Ct \lambda^\alpha)$ for all $ \lambda >0$. Let $\tau$ denote the hitting time of $-1$ of this process. Notice that $\tau<\infty$ a.s.\ because $\zeta$ does not drift to infinity, and we have $\zeta_\tau = -1$ since $\zeta$ has no negative jumps. We write $\Dzeta_\tau$ for the infinite vector consisting of the sizes of the jumps of $\zeta$ before time $\tau$, ranked in decreasing order. Then we define a probability distribution $\nu_\alpha$ on $ (\real_+)^{\natural^*}$ by
$$	\int \dd\nu_\alpha( \mathbf{x}) F( \mathbf{x}) =  \frac{\exptm{\frac1\tau F(\Dzeta_\tau)}}{\exptm{\frac1\tau}}.
$$
By the scaling property of stable L\'evy processes, the above definition of $\nu_\alpha$ does not depend on the constant $C>0$ appearing in the normalization of $\zeta$. 

We also define the function $\phi_\alpha(\theta) := \exptm{ \sum_{ i \geq 1} (Z_\alpha(i))^\theta }$, which we call the \emph{Biggins transform} of the multiplicative cascade $Z_\alpha$ after the seminal work of Biggins \cite{Big77}. We can now state our main result:

\begin{theorem}[Convergence of the perimeter cascade]  \label{th:main}
Let $\chi\0p$ be the random labeling obtained on the Ulam tree when the underlying loop-decorated quadrangulation is distributed according to $ \prob\0p_{(n;g,h)}$. Then we have the following convergence in distribution
\begin{eqnarray*}  
\frac1p	\mB({ \chi\0p(u)  }_{u\in\U} & \xrightarrow[p\to\infty]{(d)} & 		
		\mB({ Z_\alpha(u) }_{u\in\U}
\end{eqnarray*}
in $\ell^\infty(\U)$, where $Z_\alpha$ is the multiplicative cascade with offspring distribution $\nu_\alpha$. In addition, the Biggins transform of the multiplicative cascade $Z_\alpha$ is explicit and equals 
$$  \phi_\alpha(\theta)=  \frac{\sin(\pi (2-\alpha))}{\sin (\pi (\theta-\alpha))} 
\quad  \mbox{for } \theta \in (\alpha, \alpha+1) \quad \mbox{ and } \quad \phi_\alpha(\theta)=\infty \mbox{ otherwise}.$$
\end{theorem}

\begin{remark}
Here $\ell^\infty(\U)$ is defined as usual as the set of bounded functions on the countable set $\U$, endowed with the supremum norm.
The above convergence is much stronger than the convergence of \emph{finite dimensional marginals}, i.e.\ the weak convergence under the product topology of $\real^\U$. Roughly speaking, $\ell^\infty(\U)$ convergence implies that that there are no microscopic loops at some generation which contain macroscopic loops at a next generation. This is needed in particular to ensure that the convergence is preserved under relabelling of the loops. For example, if we consider the auxiliary process $(\widetilde \chi\0p(n,i))_{n\in\N,\,i\in\N^*}$, where $\widetilde \chi\0p(n,i)$ is the half-perimeter of the $i$-th largest loop at the $n$-th generation, then the finite dimensional convergence of $(p^{-1}\chi\0p(u))_{u\in\U}$ would not be enough to imply finite dimensional convergence of $(p^{-1}\widetilde \chi\0p(n,i))_{n\in\N,\,i\in\N^*}$, but $\ell^\infty(\U)$ convergence  does imply it (and furthermore implies $\ell^\infty(\U)$ convergence of $(p^{-1}\widetilde \chi\0p(n,i))_{n\in\N,\,i\in\N^*}$).
\end{remark}

\paragraph{Properties of the multiplicative cascade $Z_\alpha$.}
In Section~\ref{sec:properties} we establish some interesting properties of the multiplicative cascade $Z_\alpha$. First, in Section~\ref{sec:biggins_transform}, we define the family of additive martingales, an important observable of the multiplicative cascade. We also calculate the rate function of the multiplicative cascade, i.e.~the Legendre--Fenchel transform of $\log \phi_\alpha$. In Section~\ref{sec:volume}, we study the \emph{Malthusian martingale}
$$ \sum_{|u|=n} \left(Z_{\alpha}(u)\right)^{\min(2,2\alpha-1)}.$$
We show that it is uniformly integrable and identify the law of its limit to be equal to (in the dilute phase $\alpha>3/2$) or related to (in the dense phase $\alpha < 3/2$) an inverse-Gamma distribution with explicit parameters. As explained there, one can prove that this distribution is the scaling limit of the volume of a critical $O(n)$-decorated map with a boundary, assuming that the family of renormalized volumes is uniformly integrable.
Finally, in Section~\ref{sec:Lp}, we establish $L^p$-convergence of the additive martingales, for suitable $p$. This ensures that the multiplicative cascade $Z_\alpha$ displays no pathological behavior.

We now outline the proof of Theorem \ref{th:main}, which comes in three parts:

%
%

\paragraph{Convergence of finite dimensional marginals.}
It is proved in \cite{BBG11} that a loop-decorated quadrangulation $\ql$ distributed according to $\prob\0p_{(n;g,h)}$ can be split into its gasket---the part of $\mathfrak{q}$ outside the outer-most loops of $\boldsymbol\ell$---and a number of smaller loop-decorated quadrangulations which, conditionally on the gasket, are independent and follow the same type of distribution as $\ql$.
This settles the Markovian branching structure of the perimeter process $\chi\0p$, thus reducing the problem of convergence of its finite dimensional marginals essentially to the convergence of its first generation (Proposition \ref{prop:1ere gen}).

It is also shown in \cite{BBG11} that the gasket is a bipartite Boltzmann map with a boundary where each face of degree $2k$ receives a weight $g_k$ (which is a simple function of $F_k(n;g,h)$, see \eqref{eq:gkdugassket}).
This random map model (more precisely, the pointed version of it, see below) has been introduced in \cite{MM07} and studied under these hypotheses in \cite{LGM09}: Applying (a variant of) the classical Bouttier--Di Francesco--Guitter bijection \cite{BDFG04}, the (pointed) gasket is coded by a two-type Galton-Watson forest.
The latter can be further simplified by applying a bijection of Janson \& Stef\'ansson  \cite{JS12} that transforms it into a one-type Galton--Watson forest, which under our assumption consists of $p$ i.i.d.\ Galton-Watson trees with a critical offspring distribution in the domain of attraction of the spectrally positive $\alpha$-stable distribution.
In this coding, the loops of the first generation are transformed into the (large) faces of the gasket and then into the (large) jumps of the \L{}ukasiewicz path encoding the one-type Galton-Watson forest.
The latter naturally converges to the jumps of an $\alpha$-stable L\'evy process, which explains the appearance of the process $\Dzeta_\tau$ in the definition of $\nu_\alpha$.
The above chain of transformations is summarized in a diagram at the beginning of Section \ref{sec:firstgene}.

One technical issue in this program is that the Bouttier--Di Francesco--Guitter bijection works particularly well with \emph{pointed} maps, i.e.\ maps with a distinguished vertex.
For this reason we start by applying the bijections to the pointed gasket, and only remove the distinguished point afterwards.
This amounts to biasing the pointed gasket by the inverse of its number of vertices, which in the continuous setting give rise to the bias $\tau^{-1}$ in the definition of the measure $\nu_\alpha$.

\paragraph{A formula for left-continuous random walks.} As we saw in Theorem \ref{th:main}, the multiplicative cascade $Z_\alpha$ has an explicit and rather simple Biggins transform. This formula is obtained through a new simple identity about the (biased) first moment of some additive functionals of left-continuous random walks. This identity may have further applications and so we present it here:
Let $S$ be a left-continuous random walk on $\integer$ started from 0, i.e.\ $S_n = X_1+\cdots+X_n$, where $(X_i)_{i\ge 1}$ are i.i.d.\ random variables on $\{-1,0,1,\cdots\}$.
We denote by $T_p$ the hitting time of $-p\in\integer$ by $S$. 

\begin{theorem}	\label{th:rw} Suppose that $S$ does not drift to $\infty$ i.e. that $T_{1}< \infty$ almost surely. Then for any  positive measurable function $f:\integer\to\real$ and any $p \ge 2$ we have
$$\exptm{\frac {1}{T_p-1} \sum_{i=1}^{T_p} f(X_i)} = \exptm{f(X_1)\frac p {p+X_1}}.	$$
\end{theorem}

\paragraph{From finite dimensional to $\ell^\infty(\U)$ convergence.} We now explain how we strengthen the finite-dimensional convergence of $p^{-1} \chi^{(p)}$ towards $Z_{\alpha}$ to $\ell^{\infty}( \mathcal{U})$ convergence  (see the remark after Theorem~\ref{th:main} for why this is important). One essentially needs that, with high probability, $p^{-1}\chi\0p$ is arbitrarily small outside a finite subset of $\U$, uniformly in $p$.

We first concentrate on the first $k$ generations of the tree $\U$. Using the identity in Theorem \ref{th:rw}, we compute $\E[\sum_{|u|=k} \m({p^{-1}\chi\0p(u)}^\theta]$, the discrete analogue of Biggins transform, and show that it converges to the $k$-th power of the continuous Biggins transform given in Theorem \ref{th:main} (Lemma \ref{lem:cv Lp by generation}). This yields a moment estimate on the sizes of all the loops up to generation $k$, which implies the $\ell^\infty$ convergence on the first $k$ generations (Proposition \ref{prop:convergence Uk}). In order to strenghten this to $\ell^\infty(\mathcal U)$ convergence, we 
rely on a geometric estimate on random planar maps: if we denote by $ \overline{V}(p)$ the expected volume (i.e.\ number of vertices) of a random loop-decorated quadrangulation under the distribution $\prob\0p_{n;g,h}$, then using the Markovian structure of the gasket decomposition  it is easy to check that 
$$ \sum_{|u|=k} \overline{V}(\chi\0p(u)) \quad \mbox{ is a supermartingale indexed by }k \geq 0,$$ which gives a uniform control over all generations. This additional ingredient together with recent estimates on $ \overline{V}(p)$ due to Budd \cite{BudOn} are at the core of our proof of the $\ell^{\infty}$ convergence.\medskip

\paragraph{Related works.} Understanding the geometry of planar maps decorated with a statistical physics model in one of the major goals in today's theory of random planar maps (see \cite{SheHC,GMS15,GS15a,GS15b,BLR15,Che15,GKMW16} for recent progresses on the geometry of general random planar maps decorated by a Fortuin-Kasteleyn model). Our interest in the nested cascades in $O(n)$ model on random quadrangulations was triggered by the recent work of Borot, Bouttier and Duplantier \cite{BBD16}. They study (in the case of triangulations) in great detail the number of loops separating the boundary from a typical point; in our context, this roughly speaking consists in estimating the length of a typical branch of the tree coded by the cascade $(\chi^{(p)})$. Our perspective here is different since we study the full nested tree (rather than one branch) in a scaling limit point of view (rather than in the discrete setting). In Appendix~\ref{sec:KPZ} we provide more details of the relation between the two approaches. We also give an alternative explanation of the relation with the statistics of the number of loops surrounding a small Euclidean ball in a conformal loop ensemble.

This work obviously builds upon \cite{BBG11} where the gasket decomposition was introduced and used to study the phase diagram of Fig. \ref{fig:phase diagram}. Our study of the gasket in the non-generic critical case also borrows a lot from \cite{LGM09} and indeed the law $\nu_{\alpha}$ can be interpreted as the sizes of large faces in what would be a ``stable map with a boundary''. See also \cite{BC16} for a geometric study of the duals of the above planar maps.

It was recently shown by Gwynne and Sun \cite{GS15b} that random planar maps decorated with a Fortuin--Kasteleyn statistical mechanics model (which naturally defines an ensemble of loops) converge in the so-called peanosphere topology to the Liouville quantum sphere introduced by Duplantier, Miller and Sheffield \cite{matingoftrees}, together with an independent conformal loop ensemble. This topology allows in particular to measure the lengths of the loops in the ``quantum metric'', as well as the ``quantum volume'' of their interior. A well-known conjecture stipulates that the $O(n)$-decorated quadrangulations considered in this paper converge to the Liouville quantum disk with parameter $\gamma = \sqrt{\min(\kappa,16/\kappa)}$, also introduced in \cite{matingoftrees}, together with an independent $\mathrm{CLE}_\kappa$ in the disk. Here, $\kappa\in (8/3,8)$ is related to our parameter $\alpha$ by
 \begin{equation}
  \label{eq:alpha_kappa}
  \alpha - \frac 3 2 = 4/\kappa - 1,
 \end{equation}
so that $\kappa \in (8/3,8) \backslash \{4\}$.
In fact, our Theorem~\ref{th:main} has an analogue in the continuum, which we formulate as follows:

{\em
Consider a Liouville quantum disk with parameter $\gamma = \sqrt{\min(\kappa,16/\kappa)}$ conditioned on having quantum boundary length 1 and an independent CLE$_\kappa$ 
in the disk, with $\kappa \in (8/3,8) \backslash \{4\}$.
Then 
the nesting cascade of the quantum lengths of the CLE loops has the same law as the multiplicative cascade $Z_\alpha$ introduced in this article where $\alpha$ is given by  \eqref{eq:alpha_kappa}.}

In fact, recent work of Miller, Sheffield and Werner \cite{MSWCLEpercolation} together with a different representation of the reproduction law $\nu_\alpha$ which can be derived from \cite{Bertoin2016} yields this statement.

 \medskip 

\noindent \textbf{Acknowledgments:} We thank the Newton Institute for its hospitality during the ``Random Geometry'' program in 2015 where this work started. We acknowledge also the support of Agence Nationale de la Recherche via the grants ANR Liouville (ANR-15-CE40-0013), ANR GRAAL (ANR-14-CE25-0014) and ANR Cartaplus (ANR-12-JS02-001-01). We thank Ga\"etan Borot and J\'er\'emie Bouttier for discussions about \cite{BBG12} and Christophe Garban and Jason Miller for discussions about CLE and Liouville quantum gravity. We are also very grateful to Timothy Budd for sharing his work in progress \cite{BudOn} with us and for useful comments related to Theorem~\ref{th:volume}.

\tableofcontents
\medskip

\section{Convergence of the first generation}
\label{sec:firstgene}

\newcommand*{\Bol}{\mathfrak{B}\0p_\g}
\newcommand*{\Bolp}{ \mathfrak{B}^{(p), \bullet}_{ \g}}
\newcommand*{\FBDG}{\mathfrak{F}\0p_\mathrm{BDG}}
\newcommand*{\FJS}{\mathfrak{F}\0p_\mathrm{JS}}
\newcommand*{\Deg}{\mathsf{Deg}^\downarrow}
\newcommand*{\test}[1]{\varphi\left(#1\right)}
\newcommand*{\nb}[1]{\#\mathsf{#1}}

The goal of this section is to prove the convergence of the first generation of $\chi\0p$, as stated in the following proposition:

\begin{proposition}\label{prop:1ere gen}
For any $\varphi: \ell^\infty( \mathbb{N}^{*}) \to \real$ bounded and continuous, we have
\begin{eqnarray*}
\exptm{ \test{p^{-1}\chi\0p(i) : i \geq 1} } &\cvg{}{p\to\infty}&  \expt [\test{Z_\alpha(i) : i \geq 1}].
\end{eqnarray*}
\end{proposition}
\noindent We will follow the scheme outlined in the introduction, which is summarized in the following diagram.

\vspace{1ex}
\begin{figure}[h!]
\includegraphics[scale=0.8,page=1]{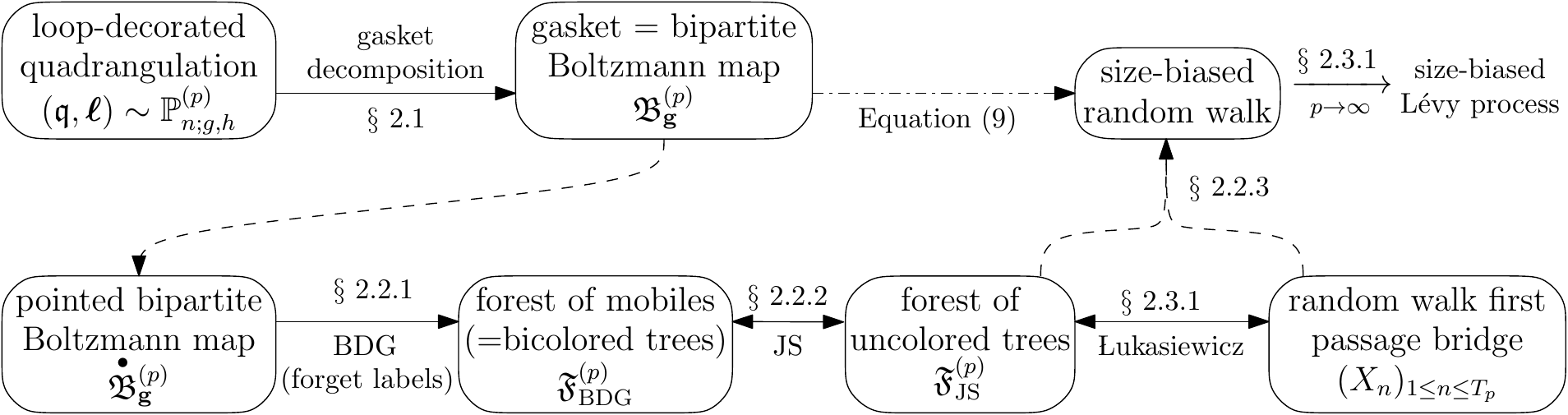}
\caption{A solid arrow $A\to B$ indicates a transformation that encodes $A$ by $B$. It is a bijection when the arrow is two-headed. A dashed arrow $A\dashrightarrow B$ indicates that the random object $B$ is obtained by size-biasing the law of $A$. Equation \eqref{eq:apresluk} is obtained by composing all these transformations and size-biasing relations. Using it, we deduce Proposition \ref{prop:1ere gen} from the classical convergence of random walks to L\'evy processes.}
\end{figure}
\vspace{-1em}

\subsection{The gasket decomposition}\label{sec:gasket}

We first recall the gasket decomposition of \cite{BBG11}.
Given a loop-decorated quadrangulation $\ql\in\Op$, let $l \geq 0$ be the number of outer-most loops in $\boldsymbol\ell$, i.e.\ loops which can be reached from the boundary of $\mathfrak{q}$ without crossing any other loop. The gasket decomposition consists in erasing all the outer-most loops and all the edges crossed by these loops. This disconnects the map into $l+1$ connected components:
\begin{itemize}
\item The \emph{gasket} is the connected component containing the external face of $ \mathfrak{q}$. This is a rooted bipartite planar map (without loops) with a boundary of length $2p$.
An internal face of the gasket is either a quadrangular face inherited from $\mathfrak{q}$, or one of the $l$ \emph{holes} obtained by removing the outer-most loops and their interior component.
Notice that a hole may have a non-simple boundary.
See Fig.~\ref{fig:taking-gasket}.

\item The $l$ remaining connected components are contained in the holes. More precisely, inside a hole of degree $2p'$, we find an element $(\mathfrak{q}', \boldsymbol{\ell}') \in\mathcal{O}_{p'}$.
\end{itemize}
\begin{figure}
\begin{center}
\includegraphics[width=15cm]{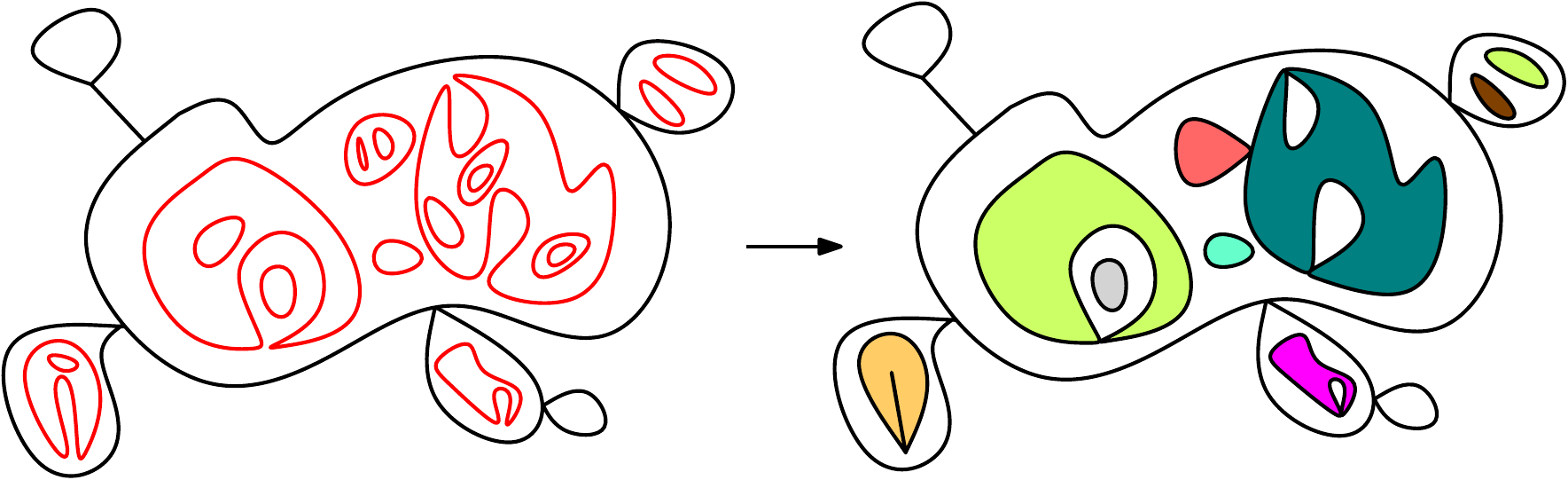}
\end{center}
\vspace{-1em}
\caption{Illustration of the gasket decomposition. Notice that after removing the outer-most loops and their interiors, the holes in the gasket may be non-simple faces. See also \cite[Fig.~4]{BBG11}}
\label{fig:taking-gasket}
\end{figure}
To be rigorous, we need to specify a root edge for each internal quadrangulation $(\mathfrak{q}',\boldsymbol\ell')$.
This can be done in a deterministic way thanks to the lack of automorphisms of rooted planar maps.
Similarly, the holes can also be numbered from $1$ to $l$ in a deterministic fashion.
Therefore, given a bipartite map $\mathfrak{b}$ of perimeter $2p$, a loop decorated quadrangulation $\ql$  that admits $\mathfrak{b}$ as gasket can be recovered by gluing a loop-decorated quadrangulation in $\mathcal{O}_k$---wrapped in a ``collar'' of $2k$ quadrangles traversed by a loop---into each face of $\mathfrak{b}$ degree $2k$ (or, if $k=2$ we have the additional possibility of gluing a plain quadrangle), see Figures 3 and 4 in \cite{BBG11}.

Then it follows from \eqref{eq:weights} that the gasket of a loop-decorated quadrangulation of distribution $\mathsf{w}_{(n;g,h)}$ is distributed according to the so-called $\g$-Boltzmann measure (see \cite{MM07})  on bipartite maps defined as
$$ w_{ \mathbf{g}}( \mathfrak{m}) = \prod_{ f \in \mathsf{Faces}( \mathfrak{m}) \backslash \{ f_{ \mathrm{r}}\} } g_{ \mathrm{deg}(f)/2},$$
where in our case the weight sequence $ \mathbf{g}=(g_{k})_{k \geq 1}$ is related to the $O(n)$ model by the relations 
\begin{equation}\label{eq:gkdugassket}
g_k = g\delta_{k,2} + nh^{2k} F_k(n;g,h),
\end{equation}
see \cite[Eq.~(2.3)]{BBG11}. If the weight sequence $ \mathbf{g}$ is such that for every $p \geq 1$, the total $ w_{ \mathbf{g}}$-mass of bipartite maps with perimeter $2p$ is finite, then the above $ \mathbf{g}$-Boltzmann measure can be normalized to define a random $ \mathbf{g}$-Boltzmann map $\Bol$ with perimeter $2p$. This is clearly implied in our context by the admissibility of the parameters $(n;g,h)$.  In the next section, we recall classical codings of Boltzmann maps (not necessarily related to the gasket of $O(n)$-decorated quadrangulations) via random labeled forests.

\subsection{Coding of bipartite Boltzmann maps with a boundary}
The coding of bipartite Boltzmann planar maps via the Bouttier--Di Francesco--Guitter (BDG) bijection  \cite{BDFG04} and the study of the induced distribution on random planar trees has been studied in depth in \cite{MM07} and more recently in \cite{BM15}. We shall recall the necessary background here referring to \cite{BM15} for details. To present the coding in its simplest form, we have to deal with \emph{pointed}  planar maps rather than  maps.\medskip 

\subsubsection{BDG coding}
A \emph{pointed map} is a map given together with a distinguished vertex $\rho$.
Let $( \mathfrak{m}, \rho)$ be a pointed bipartite map with a boundary of degree $2p$. A slight variation \cite[Section 3.3]{BM15} of the classical BDG bijection in the context of pointed bipartite maps is defined as follows.

\begin{enumerate}
\item
Draw a vertex in each face of $\mathfrak{m}$ (including the external face).
The new vertices are considered black ($\bullet$) and the old ones white ($\circ$).
Label each white vertex by its distance to the distinguished vertex $\rho$.
Since the map is bipartite, the labels of any two adjacent vertices differ exactly by one.
\item
For a face $f$ of $\mathfrak{m}$ and a white vertex adjacent to $f$, link the white vertex to the black vertex inside $f$ if the next white vertex in the clockwise order around $f$ has a smaller label.
\item
Remove the edges of $\mathfrak{m}$ and the vertex $\rho$.
It can be shown that the resulting graph is a tree \cite{BDFG04}.
\item
Let $v_0$ be the black vertex corresponding to the external face of $\mathfrak{m}$.
By removing $v_0$ and its adjacent edges, we obtain a forest of cyclically ordered trees, rooted at the neighbors of $v_0$.
Finally, we choose uniformly at random one of the trees to be the first one, and subtract the labels in all trees by a constant so that the label of the root vertex of this first tree becomes zero.
\end{enumerate}
With a moment of thought on the Step 2 of the above construction, one observes that
\begin{enumerate}[label=(\roman*)]
\item	Each internal face of degree $2k$ in $\mathfrak{m}$ gives rise to a black vertex of degree $k$ in the forest, and the forest is composed of $p$ trees.
\item	Given a black vertex of degree $k$, the possible labels on its (white) neighbors are exactly those which, when read in the clockwise order around the black vertex, can decrease at most by 1 at each step. 
If the label of one neighbor is fixed, then there are exactly ${2k-1 \choose k-1}$ possible labelings of the other neighbors which satisfy the above constraint \cite[Proof of Proposition 7]{MM07}.
\end{enumerate}
A \emph{mobile} is a rooted plane tree whose vertices at even (resp.\ odd) generations are white (resp.\ black).
We say that a forest of mobiles $(\mathfrak{t}_1,\cdots,\mathfrak{t}_p)$ is \emph{well-labeled} if (a) the root vertex of $\mathfrak{t}_1$ has label 0, (b) the labels satisfy the constraint in the observation (ii) above, and (c) the labels of the roots of $\mathfrak{t}_1,\cdots,\mathfrak{t}_p$ satisfy the similar constraint.
See \cite[Section 6.1]{BM15} for more details and a construction of the inverse mapping.

\begin{figure}
\begin{center}
\includegraphics[scale=1.2]{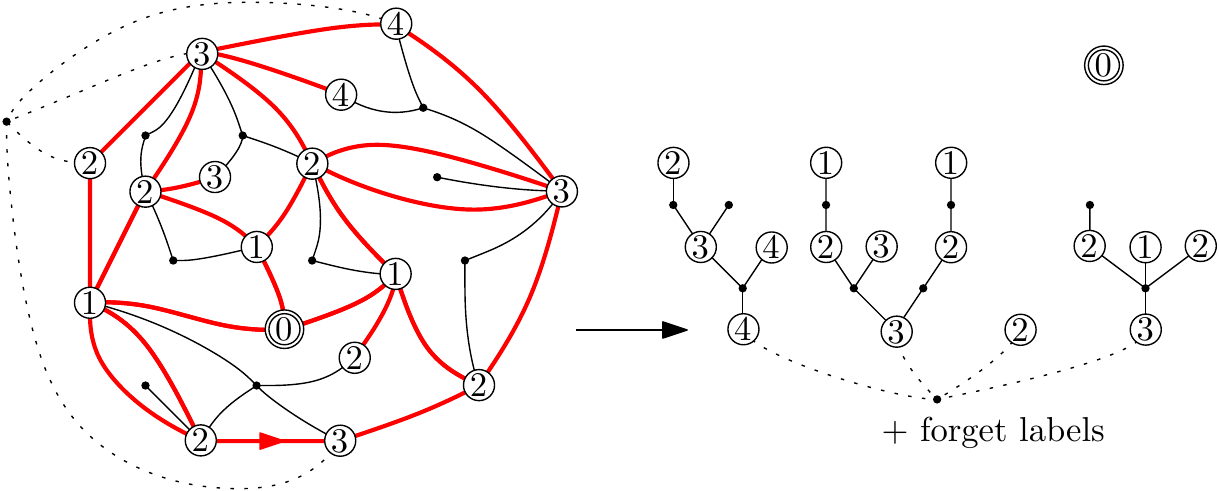}
\end{center}
\caption{Illustration of the construction of a forest of $p$ mobiles from a pointed bipartite planar map with a boundary of perimeter $2p$. The first mobile in the forest is not specified by the map and is chosen uniformly at random among all the mobiles.}
\label{fig:bdg}
\end{figure}

\medskip 

Let us now describe the effect of this coding on the Boltzmann measure. Given a weight sequence $\g=(g_k)_{k\ge 1}$, the definition of the $  \mathbf{g}$-Boltzmann measure $ w_{ \mathbf{g}}$ naturally extends to \emph{pointed} planar maps with the same formula. We suppose as above that the $w_{ \mathbf{g}}$-mass of bipartite maps with a given perimeter is finite. This implies in particular that the $w_{ \mathbf{g}}$-mass of all \emph{pointed} bipartite maps with a given perimeter is also finite. (This can be deduced from (3.2) in \cite{BBG11} and its pointed analogue. See \cite[Corollary 23]{CurPeccot}.) In these equivalent cases the weight sequence $ \mathbf{g}$ is called \emph{admissible}. (This should not be confounded with the admissibility of a triple $(n,g,h)$. The latter implies that the former when $\mathbf{g}$ is defined by \eqref{eq:gkdugassket}, but the inverse is not obvious \cite{BudOn}.) Under this assumption we can define a random bipartite map $ \mathfrak{B}_{ \mathbf{g}}^{(p), \bullet}$ with a boundary of perimeter $2p$ by normalizing the above Boltzmann measure.

If we let $\FBDG$ be the unlabeled forest of mobiles obtained by applying the construction 1.--4.\ to $\Bolp$, then it follows from the observations (i) and (ii) that $\FBDG$ is also Boltzmann distributed, with a weight 1 for white vertices and a weight
$\tg_k = {2k-1 \choose k-1} g_k$
for each black vertex of degree $k$. 
More precisely,
$$ \prob(\FBDG = \mathfrak{f})
\propto \prod_{v\in \bullet(\mathfrak{f})} \tg_{\deg(v)}. $$
where $\bullet(\mathfrak{f})$ is the set of black vertices of $\mathfrak{f}$, and the probability measure is normalized over all forests of $p$ finite mobiles. It has been shown in \cite[Proposition 7]{MM07} that $\FBDG$ is a two-type Galton--Watson forest whose law is given explicitly in terms of $\g$. For our purpose (which is proving Proposition \ref{prop:1ere gen}), we could redo the classical analysis of Galton--Watson trees in this context of multi-type Galton--Watson trees but we use a much quicker road using a recent trick discovered by Janson \& Stef\'ansson \cite{JS12}.

\subsubsection{Janson \& Stef\'ansson's trick}
 
In \cite[Section 3]{JS12}, Janson \& Stef\'ansson discovered a mapping which transforms a mobile into a rooted plane tree by keeping the same set of vertices, but changing the set of edges so that every white vertex is mapped to a leaf, and every black vertex of degree $k$ is mapped to an internal vertex with $k$ children.
We refer to \cite[Section 3.2]{CKperco} for details of this transformation.
The curious reader may have a look at the figure below and try to guess how the bijection works.
\begin{figure}[!h]
\begin{center}
 \includegraphics[width=14cm]{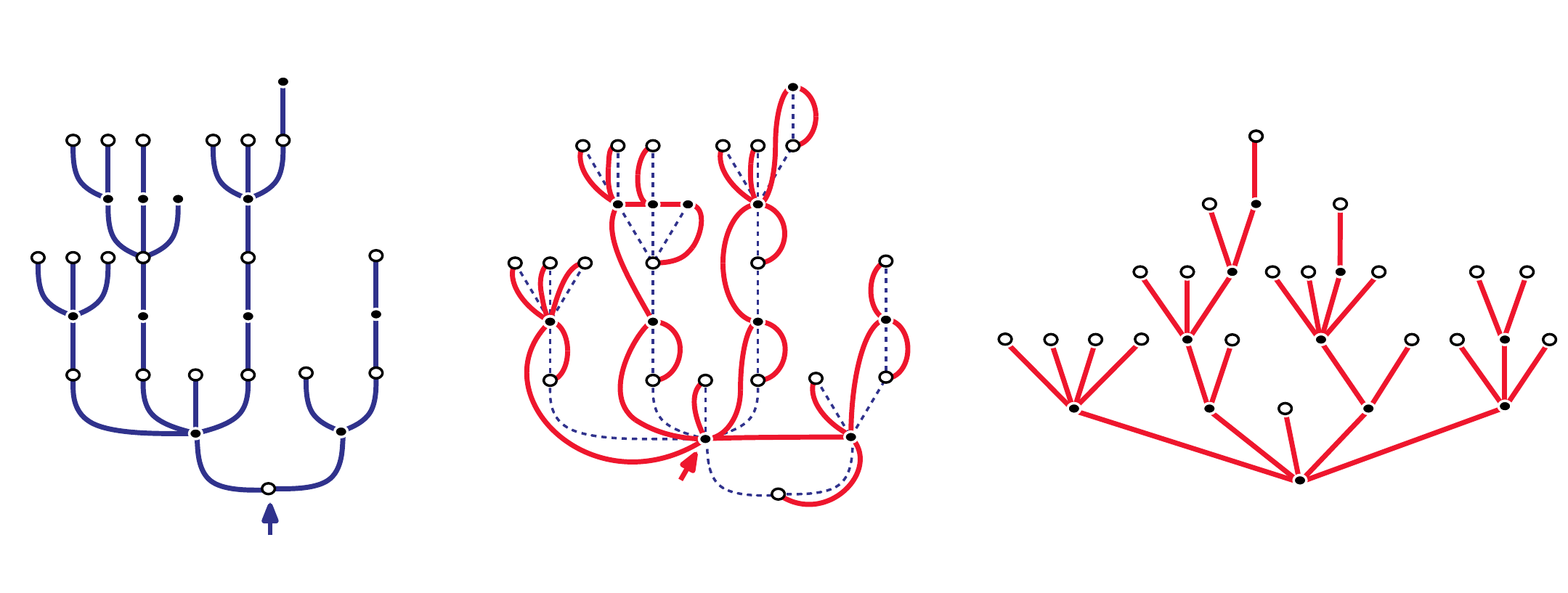}
 \caption{Illustration of the Janson \& Stef\'ansson transformation.}
\end{center}
\end{figure}

The usefulness of this bijection is that the two-type Galton--Watson trees which arise from the BDG bijection in the last section are transformed into a (one-type) Galton--Watson tree.
In our setting, let $\FJS$ be the image of $\FBDG$ under this bijection. The next proposition gathers and summarizes \cite[Proposition 1]{MM07} and \cite[Appendix]{JS12} (see also \cite[Proposition 3.6]{CKperco}) in our context:
\begin{proposition}\label{prop:law JS}
The weight sequence $\g=(g_k)_{k\ge 1}$ is admissible if and only if the equation on $x$
$$ \varphi_\g(x) := \sum_{k\ge 1} \tilde{g}_k x^k = x-1	$$
has a positive solution.
For an admissible $\g$, let $Z_\g$ be the smallest such solution, then $\FJS$ is a forest of $p$ i.i.d.\ Galton--Watson trees with offspring distribution
\begin{equation}\label{eq:def mujs}
	\mujs(0) = \frac{1}{Z_\g}
\qquad\text{and}\qquad
	\mujs(k) = \tilde{g}_k Z_\g^{k-1}\ ,\quad k\ge 1.	
\end{equation}
\end{proposition}

With this one can (at last) state the definition of non-generic criticality that we imposed in this paper (this obviously matches the definition given in \cite{BBG11}  and in \cite{LGM09} for the associated $ \mathbf{g}$ sequence, see also \cite[Section 5.1.3]{CurPeccot}):

\begin{definition} \label{def:non-generic} A weight sequence $ \mathbf{g}$ is called critical and non-generic with exponent $\alpha \in (1,2)$ if the sequence 
$ \mathbf{g}$ is admissible for the pointed bipartite map model and if the offspring distribution $ \mu_{ \mathsf{JS}}$  is critical (i.e.\ it has mean one) and satisfies 
$$ \mu_{ \mathsf{JS}}(k) \sim C k^{-\alpha-1}, \quad \mbox{ as }k \to \infty, \quad \mbox{ for some }C>0.$$
By extension, a triplet $(n;g,h)$ with $n \in (0,2)$ and $g,h \geq 0$ is called critical and non-generic if the associated $ \mathbf{g}$ sequence defined by \eqref{eq:gkdugassket} is critical and non-generic.
\end{definition}
Notice in particular that in the above definition $ \mu_{ \mathsf{JS}}(k)$ has a polynomial tail. When the weight sequence $ \mathbf{g}$ is derived from $(n;g,h)$ as in \eqref{eq:gkdugassket} this means that the exponential factors in \eqref{eq:F_p asymptotic}, \eqref{eq:gkdugassket} and \eqref{eq:def mujs} must cancel out to leave only the polynomial part in the asymptotic of $ \mu_{ \mathsf{ JS}}(k)$, see \cite[Section 3.3]{BBG11}.

\subsubsection{Back to the non-pointed gasket}
In the last section, we have recalled the coding of \emph{pointed} Boltzmann maps via random trees. 
To come back to non-pointed maps, we need to bias the law of the pointed map by the inverse of its number of vertices.
Notice that under the BDG bijection and Janson \& Stef\'ansson's mapping, the vertices of $\Bolp$, except the distinguished one, are mapped to the white vertices of $\FBDG$, then to the leaves of $\FJS$.

To summarize this chain of transformations into one equation, let $\Deg_f$ (resp.\ $\Deg_\bullet$; $\Deg_\mathtt{out}$) denote the sequence of degrees of faces (resp.\ degrees of black vertices; numbers of children) in a bipartite map (resp.\ forest of mobiles; forest of trees), ranked in the decreasing order and padded into an infinite sequence with zeros. Recall also that $ \Bol$ (resp. $\Bolp$) is a  $ \mathbf{g}$-Boltzmann map (resp.\ pointed map) with a boundary of perimeter $2p$. 
Then for any positive measurable function $\varphi:\natural^{\natural^*}\!\to\real$ we have
\vspace{-1ex}
\begin{eqnarray}
\exptm{\test{ \frac12 \Deg_f(\Bol)}}
&=&	\frac1{ \expt[1/\, \nb{Vertex}(\Bolp)] }	
	\exptm{ \frac{ \test{\frac12\Deg_f(\Bolp)} }{ \nb{Vertex}(\Bolp) } }
\notag\\&=&	\frac1{ \expt[1/\, (1+\#\circ(\FBDG))] }
	\exptm{ \frac{ \test{\Deg_\bullet(\FBDG)} }{ 1+\#\circ(\FBDG) } }
\notag\\&=&	\frac1{\expt[ 1/\,(1+\nb{Leaf}(\FJS))]}	
	\exptm{ \frac{ \test{\Deg_\mathtt{out}(\FJS)} }{ 1+\nb{Leaf}(\FJS) }}.
\label{eq:unpointing the gasket}
\end{eqnarray}
The above chain of equality is valid for any admissible sequence $\g$.

\subsection{Scaling limit of the face degrees in the gasket}

The discussion in the previous section is valid for Boltzmann maps with a general (admissible) weight sequence, and we now consider a weight sequence $\mathbf{g}$ that is derived from a critical non-generic set of parameters $(n;g,h)$ with exponent $ \alpha \in (1,2)$ and prove Proposition \ref{prop:1ere gen} (but the results are valid for any critical non-generic weight sequence $ \mathbf{g}$ as those considered in \cite{LGM09}, \cite{BC16}, \cite[Section 5.1.3]{CurPeccot}).

\subsubsection{Random walk coding}\label{sec:randomwalkcoding}

We now use the well-known random walk coding of trees and forests to study the right-hand side of \eqref{eq:unpointing the gasket}. Let $S_n=X_1+\cdots+X_n$ be a random walk where $(X_n)_{n\ge 1}$ is an i.i.d.\ sequence with distribution
$\proba(X_1+1=k)=\mujs(k)$ for $k\ge 0$. Define the first passage time of $(S_n)_{n\ge 0}$ to the level $-p$
$$
	T_p = \inf\set{n\ge 0}{S_n=-p},
$$
and let $L_p=\sum_{i=1}^{T_p} \idd{X_i=-1}$ be the number of negative steps of the walk up to $T_p$. 
Let $\boldX$ be the sequence $(X_n+1)_{1\le n\le T_p}$ ranked in decreasing order and padded with zeros. 
The classical coding of forests by their \L{}ukasiewicz paths shows that the sequence $\boldX$ has the same law as $\Deg_\mathtt{out}(\FJS)$ and that jointly we have $\nb{Leaf}(\FJS) = L_p$ in distribution.
Therefore \eqref{eq:unpointing the gasket} can be continued into
\begin{eqnarray}
\exptm{\test{ \frac12 \Deg_f(\Bol)}}
  &=&  \frac1{\expt[ 1/\,(1+\nb{Leaf}(\FJS))]}	
	\exptm{ \frac{ \test{\Deg_\mathtt{out}(\FJS)} }{ 1+\nb{Leaf}(\FJS) }}\nonumber \\
 &=& \frac{1}{{ \mathbb{E}[1/(1+L_{p})]}} \mathbb{E}\left[\frac{\varphi\mb({ \boldX}}{1+L_{p}}\right].\label{eq:apresluk}
\end{eqnarray}

By definition, the first generation of $\chi\0p$ is the sequence of half-degrees of the \emph{holes} in the gasket, sorted in the decreasing order and padded with zeros.
Recall that the faces of the gasket are either holes or regular faces of degree 4.
Therefore, if $\g$ is the weight sequence given by \eqref{eq:gkdugassket}, then the first generation of $\chi\0p$ differs from $\frac12\Deg_f(\Bol)$ at most by 2 in the $\ell^\infty(\natural^{*})$ norm. From the last display and the fact that $L_{p} \geq p$ it follows that, for any bounded continuous function $\varphi: \ell^\infty(\natural^{*}) \to \real$, we have
\begin{eqnarray}
\exptm{\test{p^{-1}\chi\0p(i) : i\ge 1} }
&=&		\frac1{\expt[ 1/ L_{p}]}	
\exptm{ \frac{ \test{p^{-1}\cdot \boldsymbol{ \mathcal{X}}^{(p)}} }{ L_{p} }} +o(1).
\label{eq:1st gen}
\end{eqnarray}
as $p\to\infty$. With all the reductions we have been through we are now in position to prove Proposition \ref{prop:1ere gen}.

\begin{proof}[Proof of Proposition \ref{prop:1ere gen}]
Recall from \eqref{eq:def mujs} that the step distribution of the walk $S$ is supposed by Definition~\ref{def:non-generic} to be centered and in the domain of attraction of the totally asymmetric stable law of parameter $ \alpha$. Recall also the notation from the Introduction and in particular that $\zeta$ is a standard $\alpha$-stable L\'evy process with no negative jumps. We an suppose that $\zeta$ has been normalized so that by a classical invariance principle we have
$$ \m({ \frac{1}{n} S_{[t n^{ \alpha}]}}_{t \geq 0} \xrightarrow[n\to\infty]{(d)} (\zeta_{t})_{t \geq 0}$$ for the Skorokhod topology. With standard arguments, one can show that the above convergence in distribution holds jointly with (using the notation of the Introduction)
\begin{equation} \label{eq:convdeuxtrucs} 
p^{-\alpha} T_p    \xrightarrow[p\to\infty]{(d)} \tau \quad\mbox{ and }\quad
\frac{1}{p} \boldX \xrightarrow[p\to\infty]{(d)} \Dzeta
\end{equation}
where the second convergence takes place in the $\ell^\infty(\natural^{*})$ topology.
We now give a lemma controlling $L_p$ via $T_p$ in a precise manner:

\begin{lemma}\label{lem:large deviation}
There is $c>0$ depending only on the weight sequence $\g$, such that for all $\varepsilon>0$ and $p\ge 1$,
\begin{flalign}\label{eq:bound 1}
&&	\prob(T_p\le \varepsilon\,p^\alpha) 
\le \prob(L_p\le \varepsilon\,p^\alpha) 
& \le \exp\mb({ -c\, \varepsilon^{-\frac1{\alpha-1}} }&
\\\label{eq:bound 2}
&&	\prob\mB({ \Big|\frac{L_p}{T_p} - \mujs(0) \Big| \ge p^{-\alpha/4} }
& \le c^{-1} p^{\alpha/2}\exp(- c\,\sqrt{p})&
\\\label{eq:bound 3}\text{and for all }K\ge \frac2{\mujs(0)},\hspace{-3cm}
&&	\prob\mB({ \frac{T_p}{L_p} \ge K }
& \le c^{-1}\exp(- c\,K p).&
\end{flalign}%
\end{lemma}

We finish the proof of Proposition~\ref{prop:1ere gen} given Lemma~\ref{lem:large deviation}. Equation \eqref{eq:bound 2} implies that $L_p/T_p$ converges to $\mujs(0)$ in probability, thus in distribution jointly with \eqref{eq:convdeuxtrucs}. Hence we have 
$$\mujs(0) \cdot \frac{p^\alpha}{L_p}\, \varphi \mb({ p^{-1} \boldX } \quad \xrightarrow[p\to\infty]{(d)} \quad \frac1\tau\, \varphi \mb({\Dzeta}.$$
On the other hand, \eqref{eq:bound 1} implies that the sequence $(p^\alpha L_p^{-1})_{p \ge 1}$ is uniformly integrable.
Therefore we can take expectations in the last convergence in distribution and it follows that
\begin{eqnarray*}
\frac{ \exptm{ L_p^{-1} \varphi\mB({p^{-1}\boldX}  }}{ \expt [L_p^{-1}]} = \frac{ \mujs(0)\cdot \exptm{p^\alpha L_p^{-1} \varphi\mB({p^{-1}\boldX} }}{ \mujs(0)\cdot \expt[ p^\alpha L_p^{-1}]}  & \cvg{}{p\to\infty} & \frac{ \exptm{ \tau^{-1} \varphi \mb({\Dzeta} }}{ \expt[\tau^{-1}]} = \expt [\varphi(Z_\alpha(i) : i \ge 1)].
\end{eqnarray*}
With \eqref{eq:apresluk} this finishes the proof of the proposition.
\end{proof} 

\begin{proof}[Proof of Lemma \ref{lem:large deviation}]
The first inequality in \eqref{eq:bound 1} follows from $L_p\le T_p$.
For the second inequality, consider for $\lambda >0$ the non-negative martingale 
$$M_n = \exp\mB({ -\lambda S_n - \Psi(\lambda) \sum_{i=1}^{n} \idd{X_i=-1} }$$
where $\Psi(\lambda)$ is defined by the equation $\exptm{\exp(- \lambda X_1 - \Psi(\lambda) \idd{X_1=-1})} = 1$, or explicitly by
$$	\Psi(\lambda) = -\log\mB({ 1-\frac{ \expt[e^{-\lambda X_1}] -1 }{e^\lambda \proba(X_1=-1)} }.$$ 
By Fatou's lemma, $\expt[M_{T_p}] = \expt[\exp(\lambda p-\Psi(\lambda)L_p)] \le 1$. Notice that since $ \mathbb{E}[e^{-\lambda X_{1}}] \geq  e^{-\lambda \mathbb{E}[X_{1}]}=1$, we always have $\Psi(\lambda) \geq 0$ as soon as $\lambda >0$. We can thus apply the Chernoff bound and get
\begin{equation}\label{eq:Chernoff}
		\proba(L_p\le \varepsilon p^\alpha)
\ \le\	\expt\mB[{ \exp\mb({\varepsilon p^\alpha \Psi(\lambda) -\Psi(\lambda) L_p} }
\ \le\	\exp\mb({ \varepsilon p^\alpha \Psi(\lambda) - \lambda p }.
\end{equation}
From our standing assumptions, we know that $X_1$ has the power law tail behavior $\prob(X_1\ge x) \sim C x^{-\alpha}$ ($x\to\infty$) and so by standard Abelian theorems its Laplace transform witnesses the following asymptotic:
$$\expt[e^{-\lambda X_1}] = 1 + C'\lambda^\alpha + o(\lambda^\alpha).$$
It follows that $\Psi(\lambda)\sim C''\lambda^\alpha$ as $\lambda\to 0$.
On the other hand, it is easy to see that $\Psi(\lambda)\sim\lambda$ when $\lambda\to\infty$.
Therefore there exists a constant $c'$ such that $\Psi(\lambda)\le c'\,\lambda^\alpha$ for all $\lambda>0$.
Then \eqref{eq:bound 1} follows from \eqref{eq:Chernoff} by taking $\lambda = c''(\varepsilon^{\frac{1}{\alpha-1}}p)^{-1}$ with $c''>0$ sufficiently small.

For \eqref{eq:bound 2}, observe that for all $\beta>0$ and $\theta \in (0, \alpha)$,
\begin{align*}
	\prob\mB({\, \mB|{ \frac{L_p}{T_p}-\mujs(0) } \ge p^{-\beta\alpha}}
&	=	\sum_{n=1}^\infty \prob\mB({\, \mB|{ 
			\frac1 n\sum_{i=1}^n \idd{X_i=-1} -\mujs(0)
		} \ge p^{-\beta\alpha}		\text{ and }	T_p=n }
\\&	\le \prob(T_p < p^\theta) + \sum_{n=p^\theta}^\infty
	\prob\mB({\, \mB|{ 
			\frac1 n\sum_{i=1}^n \idd{X_i=-1} -\mujs(0)
		} \ge p^{-\beta\alpha} }
\\&	\underset{(*)}\le \exp\mb({ - c\,p^\frac{\alpha-\theta}{\alpha-1}}
		+ \sum_{n=p^\theta}^\infty
		2\exp( -\tilde{c}\, n\, p^{-2\beta\alpha} )
\\&	\le\, \exp\mb({ - c\,p^\frac{\alpha-\theta}{\alpha-1}}
		+ 2\,\tilde{c}^{-1}\, p^{2\beta\alpha} 
		 \exp( -\tilde{c}\, p^{\theta-2\beta\alpha} ).
\end{align*}
where for $(*)$ we used \eqref{eq:bound 1} with $\varepsilon=p^{\theta-\alpha}$ and the standard Chernoff bound for i.i.d.\ Bernoulli random variables. 
The constant $\tilde{c}$ depends only on $\mu_0$.
We obtain \eqref{eq:bound 2} by taking $\beta=1/4$ and optimizing over $\theta$.

For \eqref{eq:bound 3}, we start by observing that $L_p\ge p$, therefore $T_p\ge pK$ on the event $\m.{T_p/L_p\ge K}$. Then using the same arguments as for \eqref{eq:bound 2}, we get for all $K\ge 2/\mujs(0)$,
\begin{align*}
	\prob\mB({\, \frac{T_p}{L_p} \ge K}
&	=	\sum_{n=pK}^\infty 
			\prob\mB({\, \frac1n \sum_{i=1}^n \idd{X_i=-1} \ge K^{-1} 
						 \text{ and }	T_p=n }
\\&	\le \sum_{n=pK}^\infty 
			\exp\m({ -\tilde{c}\, (\mujs(0)-K^{-1})\,n }
\ \le\ \, \frac{\exp\m({ -\frac12\tilde{c}\,\mujs(0)\, Kp }
			}{1-\exp\m({ -\frac12\tilde{c}\, \mujs(0)}}.
\qedhere
\end{align*}
\end{proof}

\section{A formula for left-continuous random walks}

In this section, we prove Theorem~\ref{th:rw} and an analogue of it for spectrally positive L\'evy processes. 

\subsection{Proof of Theorem \ref{th:rw}} \label{sec:kemperman}

Throughout the section, we denote by $S_n = X_1+ \cdots +X_n$ a left-continuous random walk on $\integer$  (that is, $(X_i)_{i\ge 1}$ are i.i.d.~with $X_i \ge -1$) and by $T_p$ the hitting time of $-p\in\Z$ by $S = (S_n)_{n\ge0}$.
In particular we have $S_{T_p} = -p$. The proof of Theorem \ref{th:rw} will make a heavy use of Kemperman's formula:
$$ \mbox{For all }n\ge 1 \mbox{ and }p\ge 1, 
\qquad \prob(T_p=n) = \frac p n \,\prob(S_n = -p).$$
See e.g.\ \cite[Section 6.1]{Pit06} (where there the notation $S_n$ stands for our $S_n-n$).
More precisely it follows from \cite[Lemma 6.1]{Pit06} that if $n\ge 1$ and $p\ge 1$, then for any positive measurable function $F(x_1,\ldots,x_n)$ which is invariant under cyclic permutation of its arguments, we have the extended Kemperman's formula:
\[	\expt[F(X_1, \cdots ,X_n) \idd{T_p=n}] 
=	\frac p n \expt[F(X_1, \cdots ,X_n)\idd{S_n=-p}]. \]

\begin{proof}[Proof of Theorem \ref{th:rw}] Let $n \ge 2$, we have
\begin{align*}
A_n &:= \exptm{\sum_{i=1}^n f(X_i)\idd{T_p=n}}\\
  &= \frac p n\,\exptm{\sum_{i=1}^n f(X_i)\idd{S_n=-p}} && \text{by extended Kemperman's formula}\\
  &= p\, \exptm{f(X_1)\idd{S_n=-p}}&& \text{by cyclic symmetry}\\
  &= p\, \exptm{f(X_1)\prob(S_{n-1}=-p-X_1\,|\,X_1)}&& \text{by Markov property}\\
  &= \exptm{f(X_1)\frac p {p+X_1} (n-1) \prob(T_{p+X_1} = n-1\,|\,X_1)} && \text{by Kemperman's formula}.
 \end{align*}
Since $p \ge 2$, we have $T_p \ge 2$ and $T_{p+x} \ge 1$ almost surely, for every $x\in\{-1,0,1,\ldots\}$. Hence,
\begin{align*}
 \E\m[{\frac 1 {T_p-1} \sum_{i=1}^{T_p} f(X_i)} 
  &= \sum_{n=2}^\infty \frac{A_n}{n-1}\\
  &= \E\m[{f(X_1)\frac p {p+X_1} \sum_{n=2}^\infty \P(T_{p+X_1} = n-1\,|\,X_1)}\\
  &= \E\m[{f(X_1)\frac p {p+X_1}},
\end{align*}
where in the penultimate line we used the fact that $T_p<\infty$ almost surely to deduce that the sum inside the expectation is equal to $1$. This completes the proof of the theorem.
\end{proof}  

The theorem has the following generalization, whose proof is an easy extension of the above proof and left to the reader.
\begin{proposition} 
\label{prop:rw_extension}
Suppose that $(S_n)_{n\ge0}$ does not drift to $+\infty$. 
Let $m\in\N$, $f:\integer^m\to\real_+$ and $g:\bigcup_{j=1}^\infty \Z^j \to \R_+$ be symmetric measurable functions.
Then for any $p\ge 1$ we have
\begin{multline*}
\exptm{ \frac{\idd{T_p>m}}{(T_p-1)\cdots(T_p-m)} 
		\sum_{(i_1,\ldots,i_p)\in\mathcal{A}_{T_p}^m} f(X_{i_1},\cdots,X_{i_m})  g((X_j)_{j\notin\{i_1,\ldots,i_p\},\,j\le T_p} )}\\
=	\exptm{ f(X_1,\cdots,X_m) \frac{p\,\idd{T_p>m}}{p+X_1+\cdots+X_m} \E[g((X_j)_{j\le T_{p+X_{1}+\cdots+X_{p}}})] }
\end{multline*}
where $\mathcal{A}^m_{T_p}$ is the set of all ordered $m$-tuples of distinct elements ($m$-arrangements) from $\{1,\cdots,T_p\}$.
\end{proposition}

\subsection{Passing to the limit: An analogous formula for L\'evy processes}
Let now $ (\eta_t)_{t\ge0}$ be a L\'evy process with no negative jumps started from $\eta_0 = 0$. Denote by $\pi( \dd x)$ its L\'evy measure (supported on $ \real_+$) and by $\tau$ the hitting time of $-1$. We are interested in the mean intensity measure of the jumps $\Delta \eta_{t} := \eta_{t} - \eta_{t-}$ of $\eta$ up to time $\tau$. 
\begin{proposition}
\label{prop:levy} Suppose that $\tau<\infty$ almost surely. 
Let $f:\R_+^*\to \R$ be non-negative, measurable and such that $f(0)=0$. Then we have
 \[
  \E\m[{\frac 1 {\tau} \sum_{t \leq \tau} f( \Delta \eta_{t}) } = \int f(x)\frac{1}{1+x}\pi( \dd x).
 \]
\end{proposition}
\begin{remark} Notice that, somehow surprisingly, the drift and the Brownian component of the L\'evy process do not appear explicitly in the result (as long as the L\'evy process does not drift to $\infty$). However, they do affect the distribution of $\tau$ and of the jumps until time $\tau$.

Also note that Proposition~\ref{prop:levy} admits an obvious extension analogously to Proposition~\ref{prop:rw_extension}.
\end{remark}

\begin{proof}
We could of course adapt the proof of Theorem~\ref{th:rw} to the current setting (in the spirit of \cite{BCP04-FPbridges}) however, we find it shorter to simply argue by approximation. Let $S^{(n)}_k = X_{1}^{(n)}+ \cdots + X_k^{(n)}$ be a sequence of left-continuous random walks and $(a_n)_{n\ge0}$ a sequence of positive integers, such that we have 
 \begin{eqnarray} \label{eq:cvdiscret-continu}\m({a_n^{-1} S^{(n)}_{\lfloor nt \rfloor}}_{t \geq 0} \quad \xrightarrow[n\to\infty]{(d)} \quad (\eta_{t})_{t \geq 0}
  \end{eqnarray}
in distribution in the Skorokhod topology. In particular this means that $ n \prob(X^{(n)}_{1} \geq x a_{n}) \to \pi( (x, \infty))$ for all $x$ which is not an atom of $\pi$. Note also that it is always possible to perform such an approximation in such a way that the walk $S^{(n)}$ does not drift towards $\infty$. For any continuous function $f$ on $\R_+^*$ with compact support, we then have (where the equality $(*)$ is justified just below)
\begin{align*}
\exptm{\frac 1 {\tau} \sum_{t\le \tau} f(\Delta \eta_t) } & \underset{(*)}{=} \lim_{n\to\infty} \exptm{\frac n {T_{a_n}} \sum_{i=1}^{T_{a_n}} f\m({a_n^{-1} X\0n_i}}\\
&\underset{ \mathrm{Thm.} \ref{th:rw}}{=}  \lim_{n\to\infty} n \exptm{f(a_n^{-1} X_1)\frac{a_n}{a_n+X_1}}\\
&=  \int f(x)\frac{1}{1+x} \pi(\dd x).
\end{align*}
The statement then follows by a monotone class argument.
In order to justify $(*)$ one can first invoke the Skorokhod embedding theorem and assume that \eqref{eq:cvdiscret-continu} holds almost surely. It then follows from standard arguments that $\frac n {T_{a_n}} \sum_{i=1}^{T_{a_n}} f\m({a_n^{-1} X\0n_i} \to \frac 1 {\tau} \sum_{t\le \tau} f(\Delta \eta_t)$ in distribution as $n \to \infty$. It thus remains to prove uniform integrability in order to allow convergence of the expectations. Without loss of generality, we can assume that $f$ is supported in $[1,\infty)$ and bounded by $1$, that is, $f\le \idd{x\ge 1}$. 
Define $N\0n_k=\#\m.{1\le i\le k: X\0n_i\ge a_n}$,
then we can write
$$\exptm{ \m({ \frac{n}{T_{a_n}\!} \sum_{i=1}^{T_{a_n}} f\m({a_n^{-1} X\0n_i}}^2 } 
\,\le\ \exptm{ \m({ \frac{n}{T_{a_n}\!\!} N\0n_{T_{a_n}} }^2 }
\,\le\ \exptm{ \m({ \frac{n}{T_{a_n}\!\!} N\0n_n }^2\! \idd{T_{a_n}\le n} }
		+ \exptm{ \m({ \sup_{k\ge n}\frac nkN\0n_k }^2 }
$$
Since $\prob(X\0n_1\ge a_n)$ is of order $1/n$, we can choose $\lambda$ large enough so that $\prob(X\0n_1\ge a_n)\le 1-\exp(-\lambda/n)$ for all $n$.
Then the process $(N\0n_k)_{k\ge 1}$ is stochastically bounded by $(Y_{k/n})_{k\ge 1}$, where $Y$ is a standard Poisson process of intensity $\lambda$.
Easy estimates show that $\exptm{(\sup_{t\ge 1}t^{-1}Y_t)^2}<\infty$, which gives a uniform bound to the second term on the right-hand side of the last display.
For the first term, we apply Cauchy-Schwarz inequality to get that
$$\exptm{ \m({\frac n{T_{a_n}} N\0n_n}^2 \idd{T_{a_n}\le  n} }
\,\ \le\,\ \m({ \expt[Y_1^4] \cdot \exptm{ \m({\frac n {T_{a_n}}}^4 \idd{T_{a_{n}}\le n} } }^{1/2}.$$
Using estimates similar to those of Lemma \ref{lem:large deviation} we deduce that $\exptm{ (n/T_{a_n})^4 \idd{T_{a_{n}}\le  n}}$ is bounded uniformly in $n$. Gathering the estimates we deduce that $\exptm{ \m({ \frac n{T_{a_n}} \sum_{i=1}^{T_{a_n}} f\m({a_n^{-1} X\0n_i} }^2 }$ is bounded uniformly in $n$. This gives the desired uniform integrability.
\end{proof}

\section{Properties of the limiting multiplicative cascade}
\label{sec:properties}

In this section, we examine in detail the multiplicative cascade $Z_\alpha$ defined in the introduction. The most important quantity of a multiplicative cascade, which determines much of its asymptotic behavior, is the Biggins transform. We calculate this transform in Section~\ref{sec:biggins_transform}, relying on the formula for L\'evy processes proved in the previous section (Proposition~\ref{prop:levy}). This allows to define additive martingales which we make explicit. We also calculate the Legendre--Fenchel transform of the (log-)Biggins transform, which describes the asymptotic growth of the multiplicative cascade. In Section~\ref{sec:volume}, we show that the Malthusian martingale of the multiplicative cascade is uniformly integrable and calculate the law of its limit. We conjecture this law to be the asymptotic law of the renormalized volume of the $O(n)$-decorated quadrangulations considered in this paper. Finally, in Section~\ref{sec:Lp}, we study $L^p$-convergence of the additive martingales.

\subsection{The Biggins transform and additive martingales}
\label{sec:biggins_transform}

\begin{figure}[h]
 \begin{center}
 \includegraphics[width=13cm]{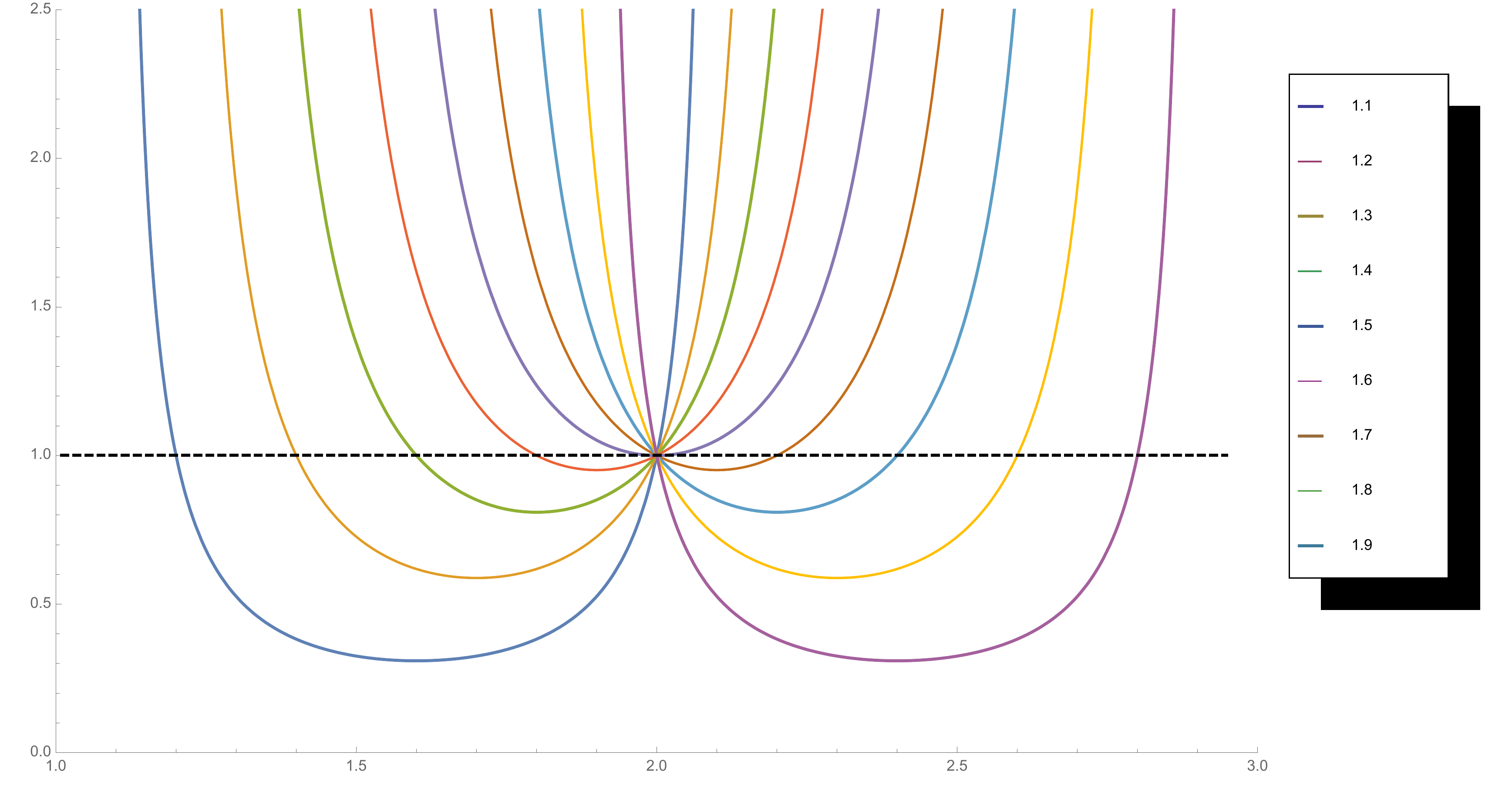}
 \caption{ \label{fig:plotage}The various functions $\phi_\alpha(\theta)$ for $\alpha = 1.1,\,1.2,\ldots,\,1.9$, and the line $y=1$ (dashed).}
 \end{center}
 \end{figure}

\label{sec:calculstable}
In this section we prove the formula for the Biggins transform of the measure $\nu_\alpha$ (see Fig.~\ref{fig:plotage} for a plot for different values of $\alpha$):
\begin{equation}\label{eq:biggins_transform}
\phi_\alpha(\theta) = \exptm{ \sum_{i=1}^\infty \mb({Z_\alpha(i)}^\theta } =  \frac{ \sin(\pi(2-\alpha))}{ \sin( \pi (\theta-\alpha))} \quad \mbox{ for }\theta \in (\alpha,\alpha+1).
\end{equation}

\begin{proof}[Proof of \eqref{eq:biggins_transform}]
We use Proposition~\ref{prop:levy} when $\eta = \zeta$ is a spectrally positive $\alpha$-stable L\'evy process for $\alpha\in(1,2)$. In this case recall that we have $ \expt [e^{-\lambda \zeta}] = \exp(C \lambda^\alpha)$ and the L\'evy measure of $\zeta$ is equal to 
$$ \pi( \dd x) =  \frac{C}{ \Gamma(-\alpha)} \frac{ \dd x}{x^{\alpha+1}} \mathbf{1}_{(x>0)}.$$
Applying the above-mentioned proposition we deduce that 
$$  \E\m[{\frac 1 \tau \sum_{t \leq \tau} ( \Delta \zeta_{t})^\theta } = \int \frac{x^\theta}{1+x}\pi( \dd x) = \frac{C}{\Gamma(-\alpha)} \int_0^\infty \frac{x^{\theta-\alpha-1}}{x+1} \dd x =   \begin{cases}
  \frac{C \pi}{\Gamma(-\alpha)\sin(\pi(\theta-\alpha))} & \tif \theta\in(\alpha,\alpha+1)\\
 +\infty & \text{ otherwise.}
 \end{cases}$$
For the last equality, see e.g.\ \cite[6.2(3)]{IntegralBook}.
Moreover, it is classical (see e.g.\,\cite[Chapter~VII, Theorem~1]{Ber96}) that for $\lambda \geq 0$ we have 
$\E[e^{-\lambda \tau}] = e^{-(\lambda/C)^{1/\alpha}}$ and so 
\begin{equation}
 \label{eq:T_biggins_transform}
 \E[1/\tau] = \E\m[{\int_0^\infty e^{-\lambda \tau}\,d\lambda} = \int_0^\infty e^{-(\lambda/C)^{1/\alpha}}\, \mathrm{d}\lambda = C\,\Gamma(1+\alpha).
\end{equation}
Combining the last two displays  with Euler's reflection formula we indeed compute the Biggins transform of the measure $\nu_\alpha$ as promised
$$ \frac{\E\m[{\frac 1 \tau \sum_{t \leq \tau} (\Delta \zeta_{t})^\theta } }{ \expt[1/\tau]} = \begin{cases}
  \frac{\sin (\pi(2-\alpha)) }{\sin (\pi(\theta-\alpha))} & \tif \theta\in(\alpha,\alpha+1)\\
 +\infty & \text{ otherwise.}
 \end{cases}$$
 We see that the above result does not depend on the normalizing constant $C$ appearing in the definition of $\zeta$. As we already remarked in the introduction, this can more directly be seen using a scaling argument to show that the law of $\nu_\alpha$ is independent of $C$.
\end{proof}

\paragraph{Additive martingales.} Consider the following family of processes, indexed by $\theta\in(\alpha,\alpha+1)$,
\begin{equation}
\label{eq:martingale}
W_n(\alpha,\theta) = \phi_\alpha(\theta)^{-n} \sum_{ u : |u|=n} (Z_\alpha(u))^\theta.
\end{equation}
It is well-known and easy to show that each of these processes is a non-negative martingale with respect to the $\sigma$-field $\F_n = \sigma(Z_\alpha(u),\,|u|\le n)$. Since for each $n\in\N$, $W_n(\alpha,\theta)$ is an additive functional of $(Z_\alpha(u))_{|u|=n}$, they are also called \emph{(the) additive martingales} of the multiplicative cascade $Z_\alpha$. Of special importance is the so-called \emph{Malthusian martingale} corresponding to the \emph{Malthusian parameter} $\theta_\alpha$ which is the smaller solution of the equation $\phi_\alpha(\theta) = 1$. One easily checks that for $\alpha\in(1,2)\backslash \{3/2\}$, there are exactly two solutions to this equation, namely $2$ and $2\alpha-1$, and so the Malthusian parameter equals
\begin{equation}
\theta_\alpha = \min(2,2\alpha-1) = \begin{cases} 2 & \text{if $\alpha > 3/2$ \quad (dilute case)}\\ 2\alpha-1 & \text{if $\alpha < 3/2$ \quad (dense case)} \label{avantrek}\end{cases}.
\end{equation}

In particular we deduce that for $\theta \in [ (2\alpha-1) \wedge 2 ,(2 \alpha-1) \vee 2]$, which  is a non-empty interval as soon as $\alpha \ne 3/2$, the multiplicative cascade $ Z_{\alpha}$ satisfies 
$$  \mathbb{E}\left[\sum_{u \in \mathcal{U}} (Z_{\alpha}(u))^{\theta} \right] = \sum_{k=0}^{\infty} (\phi_{\alpha}(\theta))^{k} < \infty,$$
in particular $(Z_{\alpha})$ belongs to $\ell^{\theta}( \mathcal{U})$ almost surely.

In Section~\ref{sec:volume}, we prove that the Malthusian martingale  $(W_n(\alpha,\theta_\alpha))_{n\ge 0}$ is uniformly integrable and identify the law of the limit. We also explain why this limit law should give the scaling limit of the volume of the $O(n)$-decorated quadrangulation with a boundary.
In Section~\ref{sec:Lp}, we provide moment bounds on $W_1(\alpha,\theta)$, which allow to prove convergence in $L^p$ of $W_n(\alpha,\theta)$ for suitable $p$ and $\theta$. 

\begin{remark} In the critical case $\alpha = 3/2$---which we do not consider in this paper---the equation $\phi_{3/2}(\theta) = 1$ has only one solution $\theta = 2$. It is well-known that in this case, the martingale $W_n(\alpha,2)$ converges to $0$ almost surely, but one can still get a non-trivial limit either by considering the so-called \emph{derivative martingale} \cite{Kyprianou1998} or by renormalizing the martingale $W_n(\alpha,2)$ appropriately \cite{AidekonShi}, the two approaches leading to the same result. 
\end{remark}

For completeness, we note that the Legendre--Fenchel transform of $\log \phi_\alpha$ can also be explicitly evaluated (we leave the calculation to the reader). This allows to determine the asymptotic growth of the multiplicative cascade, see Biggins~\cite{Biggins1979} for further details. 
\begin{proposition}
Denote by $\arccot$ the branch of the arccotangent taking values in $(0,\pi)$. For $x\in\R$,
 \begin{equation}
  \label{eq:log-biggins_transform}
  (\log \phi_\alpha)^*(x) := \sup_{\theta\in\R} \{\theta x - \log \phi_\alpha(\theta)\} = \alpha x + \tfrac x \pi \arccot(-\tfrac x \pi) - \tfrac 1 2 \log(1+\tfrac {x^2}{\pi^2}) - \log \sin(\pi(2-\alpha)).
 \end{equation}
 This function is strictly increasing and its (unique) root is negative if $\alpha \ne 3/2$ and $0$ if $\alpha = 3/2$.
\end{proposition}

\subsection{The volume of the map: the law of the Malthusian martingale limit}
\label{sec:volume}
\newcommand{\bV}{\overline{V}}

Recall the definition of the modified Bessel function of the third kind $K_\nu$ (also called Macdonald function), see e.g. \cite[Section~7.2.2]{BatemanII}:
\begin{align}
\nonumber
 K_\nu(z) &= \frac \pi {2\sin(\pi \nu)} (I_{-\nu}(z) - I_\nu(z)) \\
 \label{eq:Knu}
 &= \frac {\Gamma(\nu)\Gamma(1-\nu)} 2 \sum_{n=0} \frac 1 {n!}\left[\frac{(z/2)^{2n-\nu}}{\Gamma(n-\nu+1)} - \frac{(z/2)^{2n+\nu}}{\Gamma(n+\nu+1)}\right].
\end{align}
Recall that $\alpha\in(1,2)\backslash\{3/2\}$ throughout the paper. Define for every $\theta > 0$,
\[
 \psi_{\alpha,\theta}(x) = \frac 2 {\Gamma(\alpha-1/2)} x^{(\alpha-1/2)/\theta} K_{\alpha-1/2}(2x^{1/\theta}).
\]
Then $\psi_{\alpha,\theta}(0) = 1$ for all $\alpha$ and $\theta$, by \eqref{eq:Knu}. Note that $\psi_{\alpha,\theta}(x) = \psi_{\alpha,\theta'}(x^{\theta'/\theta})$ for every $\theta',\theta>0$. Also recall the formula \cite[7.12.23]{BatemanII}
\begin{equation}
 \label{eq:inverse_gamma_laplace}
 \psi_{\alpha,\theta}(x) = \psi_{\alpha,2}(x^{2/\theta}) = \frac{1}{\Gamma(\alpha-1/2)}\int_0^\infty e^{-x^{2/\theta} y-1/y} y^{-(\alpha+1/2)}\,dy.
\end{equation}
In particular, $\psi_{\alpha,2}$ is the Laplace transform of the inverse-Gamma distribution with parameters $\alpha-1/2$ and $1$. The following theorem identifies the law of the Malthusian martingale limit in terms of the function $\psi_{\alpha,\theta}$: 

\begin{theorem}[Law of the Malthusian martingale limit]
\label{th:volume}\
The Malthusian martignale $(W_n(\alpha,\theta_\alpha))_{n\ge 0}$ is uniformly integrable for all $\alpha \in (1,2)\setminus \{3/2\}$. Its limit $W_\infty(\alpha,\theta_\alpha)$ has the following Laplace transform:
In the dilute case $(\alpha>3/2)$, 
\begin{equation*}
\E\mb[{ e^{-x W_\infty(\alpha,2)} }  = 
\psi_{\alpha,2}((\alpha-3/2)x) \,,
\end{equation*}
that is, $W_\infty(\alpha,2)$ follows the inverse-Gamma distribution with parameters $\alpha-1/2$ and $\alpha-3/2$.
In the dense case $(\alpha<3/2)$,  
\begin{equation*}
\E\mb[{ e^{-x W_\infty(\alpha,2\alpha-1)} }  = 
\psi_{\alpha,2\alpha-1} \m({ \frac{\Gamma(\alpha+1/2)}{\Gamma(3/2-\alpha)}x }   \,.
\end{equation*}
\end{theorem}

\begin{remark}
Let $V(p)$ be the volume (i.e.\ number of vertices) of the loop-decorated quadrangulation distributed according to $\prob_{n;g,h}\0p$. We expect $W_\infty(\alpha,\theta_\alpha)$ to be the scaling limit of $V(p)$, more precisely, 
\begin{equation}\label{eq:scaling limit volume}
p^{-\theta_\alpha} V(p) \cv[]p \Lambda \cdot W_\infty(\alpha,\theta_\alpha)
\end{equation}
for some constant $\Lambda>0$. To see why, let us consider $\bV_n(p)$, the expectation of $V(p)$ conditionally on the part of the loop-decorated quadrangulation outside the loops of generation $n$. In particular $\bV_0(p) =: \bV(p)$ is the (unconditional) expectation of $V(p)$. Clearly, $(\bV_n(p))_{n\ge 0}$ is a uniformly integrable martingale that converges to $V(p)$. 

Actually, $\bV_n(p)$ is the discrete counterpart of the Malthusian martingale $W_n(\alpha,\theta_\alpha)$: According to a result of Timothy Budd (Theorem~\ref{th:budd}), $\bV(p) \sim \Lambda p^{\theta_\alpha}$ as $p\to\infty$.  Combining this with estimates on the volume of the gasket, it can be shown that the volume outside the loops of generation $n$ is negligible, hence
$\bV_n(p) \approx \sum_{|u|=n} \bV(\chi\0p_u) \approx \Lambda \cdot \sum_{|u|=n} (\chi\0p_u)^{\theta_\alpha}$.
In addition, the scaling limit of the perimeter cascade (Theorem~\ref{th:main}) gives $\chi\0p_u \approx p Z_\alpha(u)$. It follows that $\bV_n(p) \approx p^{\theta_\alpha} \Lambda \cdot \sum_{|u|=n} Z_\alpha(u)^{\theta_\alpha} = p^{\theta_\alpha} \Lambda \cdot W_n(\alpha,\theta_\alpha)$. Then, taking the limit $n\to\infty$ on both sides suggests \eqref{eq:scaling limit volume}.

The above heuristics can be turned into a rigorous proof if we assume that the family $(p^{-\theta_\alpha} V(p))_{p\ge 0}$ is uniformly integrable.
\end{remark}

\begin{remark}
The inverse-Gamma distribution is known to appear as the limiting law of the volume of planar maps decorated by statistical physics models in the dilute case. Even in the dense case, the Laplace transform $\psi_{\alpha,2\alpha-1}$ has implicitly appeared in the physics literature in the same context as this paper \cite[Equation~(2.5)]{KostovStaudacher} (we are grateful to Timothy Budd for showing this to us, this helped us find the right law!). Note that Theorem~\ref{th:volume} shows in particular that if $\alpha \in (1,3/2)$, the function $\psi_{\alpha,2\alpha-1}$ is the Laplace transform of a probability distribution, which is not obvious \emph{a priori} and for which we do not have a direct proof. In particular, we do not have an explicit expression of its density. However, this probability distribution is related to the Laplace transform of the inverse-Gamma distribution of parameter $\alpha-1/2$ by the subordination relation
\[
 \psi_{\alpha,2}(x) = \psi_{\alpha,2\alpha-1}(x^{\alpha - 1/2}) \,.
\]
\end{remark}

\begin{remark}
In order to show uniform integrability of additive martingales, one usually uses a famous result of Biggins, later improved by Lyons~\cite{Lyo1997}, which states that the martingale $W_n(\alpha,\theta)$ is uniformly integrable if $$ \theta (\log \phi_\alpha)'(\theta) < \log \phi_\alpha(\theta)\quad\text{and}\quad \E[W_1(\alpha,\theta)\log_+(W_1(\alpha,\theta))] < \infty,$$ and otherwise converges almost surely to $0$. Our proof of uniform integrability bypasses this result.
\end{remark}

The main part of the proof of Theorem~\ref{th:volume} is the following Lemma~\ref{lem:psi_equation}, which identifies the function $\psi_{\alpha,\theta}$ as the solution to a certain functional equation. Together with a general result on multiplicative cascades (Proposition~\ref{prop:mult_cascade_limit} below), this will readily imply the theorem.

\begin{lemma}
\label{lem:psi_equation}
For every $\alpha\in(1,2)$, $\theta > 0$, the function $\psi_{\alpha,\theta}$ satisfies the equation
\begin{equation}
\psi_{\alpha,\theta}(x) = \E\left[\prod_{i=1}^\infty \psi_{\alpha,\theta}(x Z_\alpha(i)^\theta)\right],\,x> 0,\quad \psi(0) = 1.
\label{eq:fixed_point_psi}
\end{equation}
\end{lemma}

\begin{proof}
It is enough to prove the formula for $\theta = 1$, by the relation $\psi_{\alpha,\theta}(x) = \psi_{\alpha,1}(x^{1/\theta})$, $x>0$. We therefore assume $\theta=1$ from now on and write $\psi := \psi_{\alpha,1}$.

We start by expressing the right-hand side of \eqref{eq:fixed_point_psi} in terms of the jumps of an $\alpha$-stable L\'evy process: Let $(\zeta_t)_{t\ge0}$ be an $\alpha$-stable L\'evy process with no negative jumps started at 0, more precisely we assume that its cumulant is given by
\begin{equation}
\log \E[e^{-\lambda \zeta_1}] = \int_0^\infty (e^{-\lambda x} - 1 + \lambda x)\frac{1}{x^{\alpha+1}}\,dx.
\label{eq:cumulant}
\end{equation}
Let $\tau$ denote the hitting time of $-1$ of $\zeta$. Then \eqref{eq:fixed_point_psi} reads,
\begin{equation}
\psi(x) = c_\alpha \E\left[\frac1\tau \prod_{t<\tau} \psi(x \Delta \zeta_t) \right],\quad c_\alpha := \left(\E\left[\frac 1 \tau\right]\right)^{-1},
\label{eq:rhs_zeta}
\end{equation}
where the product is over all jump times $t$ less than $\tau$.
By  \eqref{eq:T_biggins_transform} and Euler's reflection formula, 
\begin{equation}
c_\alpha = -\frac{\sin(\pi \alpha)}{\pi} = \frac{\sin(\pi (\alpha-1))}{\pi}.
\label{eq:1tau}
\end{equation}
Now derive \eqref{eq:rhs_zeta} with respect to $x$, which gives by the product formula,
\begin{align*}
 \psi'(x) = c_\alpha \E\left[\frac 1 \tau \sum_{t < \tau} (\Delta \zeta_t)\psi'(x\Delta \zeta_t) \prod_{s<\tau,\,s\ne t}\psi(x \Delta \zeta_t)\right]. 
\end{align*}
We the extension of Proposition~\ref{prop:levy} mentioned after its statement, we calculate this as follows:
\begin{equation}
 \label{eq:psiprime}
  \psi'(x) = c_\alpha \int_0^\infty \frac{dy}{y^{\alpha+1}} y\psi'(x y)\frac 1 {1+y} \E\left[\prod_{t<\tau} \psi(x \Delta \zeta_t)\right]^{1+y}.
\end{equation}

In order to calculate the expectation on the right-hand side of \eqref{eq:psiprime}, we use the fact that the functional $\left(\prod_{t<\tau} \psi(\lambda \Delta \zeta_t)\right)$ induces a change of measure of the L\'evy process $\zeta$, which turns it into a non-conservative L\'evy process, i.e.~a L\'evy process with killing. More precisely, define the subprobability measure $\P_{\psi,x}$ by
\[
\E_{\psi,x}[H_t] = \E\left[H_t \prod_{s<t} \psi(x \Delta \zeta_s) \right],
\]
for every $\sigma((\zeta_s)_{0\le s\le t})$-measurable bounded r.v. $H_t$. It is a standard fact that under $\P_{\psi,x}$, the process $\zeta$ is again a L\'evy process with cumulant $\kappa_{\psi,x}$ given by
\begin{equation}
\kappa_{\psi,x}(\lambda) = \int_0^\infty (e^{-\lambda y}\psi(x y) - 1 + \lambda y)\frac{1}{y^{\alpha+1}}\,dy.
\label{eq:psi_psi}
\end{equation}
It follows from the definition of $\kappa_{\psi,x}$ that it is a continuous, strictly increasing function on $[0,\infty)$. We may thus define its inverse $\kappa_{\psi,x}^{-1}$  on $[\kappa_{\psi,x}(0),\infty)$ (note that $\kappa_{\psi,x}(0)\le 0$, so in particular, $\kappa_{\psi,x}^{-1}(0)$ is well defined). The following is well-known:
\begin{equation}
\E\left[\prod_{t<\tau} \psi(x \Delta \zeta_t)\right] = \E_{\psi,x}[1] = e^{-\kappa_{\psi,x}^{-1}(0)}.
\label{eq:inverse}
\end{equation}
It turns out that $\kappa_{\psi,x}^{-1}(0) = 2x$, or, equivalently, $\kappa_{\psi,x}(2x) = 0$, as can be easily checked by elementary computations, using the formula \eqref{eq:inverse_gamma_laplace} (or by a computer algebra software).
Together with \eqref{eq:inverse}, Equation \eqref{eq:psiprime} then becomes,
\[
\psi'(x) =  c_\alpha \int_0^\infty \frac{dy}{y^{\alpha}} \psi'(x y)\frac 1 {1+y} e^{-2x(1+y)}.
\]
Changing variables $y\mapsto y/x$ in the integral, this equation becomes,
\begin{align}
\label{eq:prime2}
\psi'(x) = c_\alpha e^{-2x} x^{\alpha} \int_0^\infty \frac{dy}{y^{\alpha}} \psi'(y)\frac 1 {x+y} e^{-2y}.
\end{align}
We now recall that $\psi(x) = C_\alpha x^{\alpha-1/2}K_{\alpha-1/2}(2x)$ for some $C_\alpha > 0$. Then $\psi'(x) = -2 C_\alpha x^{\alpha-1/2} K_{\alpha-3/2}(2x)$ \cite[7.11(21)]{BatemanII}. Using this identity, \eqref{eq:prime2} is equivalent to
\[
 e^{2x} \frac{K_{\alpha-3/2}(2x)}{\sqrt x} = c_\alpha \int_0^\infty  \frac 1 {x+y} e^{-2y} \frac{K_{\alpha-3/2}(2y)}{\sqrt y}\,dy,
\]
or, equivalently,
\begin{equation}
 \frac{ e^{x/2} K_{\alpha-3/2}(x/2)}{\sqrt x} = c_\alpha \int_0^\infty  \frac 1 {x+y} e^{-y} \frac{e^{y/2}K_{\alpha-3/2}(y/2)}{\sqrt y}\,dy.
\label{eq:K}
\end{equation}
Recall $c_\alpha = \sin(\pi(\alpha-1))/\pi = \sin(\pi((\alpha - 3/2) + 1/2))/\pi$. Equation \eqref{eq:K} is then a special case of a known formula for Whittaker functions, see e.g. \cite[p335]{Cuyt2008} or \cite[13.16.6]{nist}.
This finishes the proof of \eqref{eq:fixed_point_psi}.
\end{proof}

The following proposition is a general result on multiplicative cascades, for which we provide a proof for completeness. Although results of this flavor are omnipresent in the literature and its proof idea, using multiplicative martingales, is by now standard, we could not find a suitable reference in the literature working under our minimal assumptions. 

\begin{proposition}
\label{prop:mult_cascade_limit}
 Let $(Z_u)_{u}$ be a multiplicative cascade with $Z_\emptyset = 1$. Suppose that $\E[\sum_{i=1}^\infty Z_i] = 1$. Furthermore, suppose there exists a measurable function $\phi:\R_+\to[0,1]$ satisfying  $\phi(0) = 1$,  $1-\phi(x) \sim x$ as $x\to0$, and
 \[
  \forall x\ge0: \phi(x) = \E\left[\prod_{i=1}^\infty \phi(xZ_i)\right].
 \]
 Then the martingale $W_n = \sum_{|u|=n} Z_u$ is uniformly integrable and its limiting random variable has Laplace transform $\phi$.
\end{proposition}
\begin{proof}
We first note that $(W_n)_{n\ge0}$ is a non-negative martingale and thus converges a.s. to a limit $W_\infty$. It is easy to show that this implies that $\max_{|u|=n} Z_u \to 0$ a.s., as $n\to\infty$.

Now introduce for every $x\ge0$ the process $M_n(x) = \prod_{|u|=n} \phi(xZ_u)$, $n\ge0$.
 It is well-known and easy to show that for every $x\ge0$, $(M_n(x))_{n\ge0}$  is a martingale, called a \emph{multiplicative martingale} associated to the multiplicative cascade $(Z_u)_u$. It takes values in $[0,1]$ and therefore converges a.s. and in $L^1$ to a limit $M_\infty(x)$. Furthermore, since $1-\phi(x) \sim x$ as $x\to0$ by assumption and $\max_{|u|=n} Z_u \to 0$ a.s., as $n\to\infty$, 
 \[
  \log M_\infty(x) = \lim_{n\to\infty} \sum_{|u|=n} \log\phi(xZ_u) = \lim_{n\to\infty} \sum_{|u|=n} (-xZ_u) = -xW_\infty.
 \]
 This shows that for every $x\ge0$,
 \[
  \phi(x) = M_0(x) = \E[M_\infty(x)] = \E[e^{-xW_\infty}],
 \]
 Hence, $\phi$ is the Laplace transform of $W_\infty$. Moreover, since $1-\phi(x)\sim x$ as $x\to0$, the r.v. $W_\infty$ has unit expectation. Scheff\'e's lemma then gives that $W_n$ converges in $L^1$ to $W_\infty$, hence the martingale $(W_n)_{n\ge0}$ is uniformly integrable.
\end{proof}

\begin{proof}[Proof of Theorem~\ref{th:volume}]
By Lemma~\ref{lem:psi_equation} and Proposition~\ref{prop:mult_cascade_limit}, it suffices to show that $1-\psi_{\alpha,2}(x)\sim x/(\alpha-3/2)$ as $x\to0$ (if $\alpha>3/2$), and $1-\psi_{\alpha,2\alpha-1}(x) \sim (\Gamma(3/2-\alpha)/\Gamma(\alpha+1/2))x$ as $x\to0$ (if $\alpha < 3/2$). But this is an easy consequence of \eqref{eq:Knu}, noting that the second-order term as $z\to0$ is the $n=1$ term of the first sum if $\alpha > 3/2$ ($\nu > 1$), whereas it is the $n=0$ term of the second sum if $\alpha < 3/2$ ($\nu < 1$). 
\end{proof}

\subsection{\texorpdfstring{$L^p$}{Lp}-convergence of the additive martingales}
\label{sec:Lp}

The additive martingales introduced in Section~\ref{sec:biggins_transform} are important observables of the multiplicative cascade $Z_\alpha$ and it is vital to know that they do not display pathological behavior. This is ensured by the following proposition:

\begin{proposition}
 \label{prop:Lp_convergence}
 Let $p>1$ and $\theta\in(\alpha,\alpha+1)$ be such that $\log \phi_\alpha(p\theta) < p \log \phi_\alpha(\theta)$. Then the martingale $(W_n(\alpha,\theta))_{n\ge0}$ converges in $L^p$.
\end{proposition}

The proposition will follow from classical results once the following lemma is established:

\begin{lemma}
\label{lem:pth_moment}
 Let $\theta\in (\alpha,\alpha+1)$. Then,
 \[
  \E\left[(W_1(\alpha,\theta))^p\right] < \infty\quad\text{ for every $p<(\alpha+1)/\theta$}.
 \]
\end{lemma}

\begin{proof}[Proof of Proposition~\ref{prop:Lp_convergence}]
 Let $p$ and $\theta$ as in the statement of the proposition. A classical result by Biggins \cite[Theorem~1]{Biggins1992} then gives the required convergence, provided $\E\left[(W_1(\alpha,\theta))^p\right] < \infty$. But by the hypothesis on $p$ and $\theta$, we necessarily have $p\theta < \alpha+1$, since $\phi_\alpha(\lambda) = +\infty$ for $\lambda \ge \alpha+1$. Lemma~\ref{lem:pth_moment} then implies the result.
\end{proof}

\begin{proof}[Proof of Lemma~\ref{lem:pth_moment}]
 Throughout the proof, we denote by $C$ and $C_\ep$ arbitrary positive, finite constants, whose values may change from line to line. They may depend on $\alpha$ and the constant $C_\ep$ may furthermore depend on $\ep>0$ introduced later. Recall the definition of the multiplicative cascade in terms of a spectrally positive $\alpha$-stable L\'evy process $\zeta$. It is well-known that $\tau$, the hitting time of $-1$ by $\zeta$, is a positive $1/\alpha$-stable random variable. We collect some well-known estimates on its density, see e.g. \cite[Chapter~2.5]{Zolotarev1986}:
 \begin{align}
  \label{eq:tau_infty}
  \P(\tau \in dt) \sim Ct^{-1/\alpha-1},\quad t\to\infty\\
  \label{eq:tau_zero}
  \P(\tau \in dt) = \exp(-(1+o(1)) t^{-1/\alpha-1}),\quad t\to 0.
 \end{align}
 
We now bound the tail of $W_1(\alpha,\theta)$.
 Let $x\mapsto t_x$, $\R_+\to\R_+$ be an arbitrary function for the moment, whose value we will choose later on. 
 Then, for all $x\ge0$,
 \begin{align}
  \nonumber
    \P(W_1(\alpha,\theta) > x) &= C\E\left[\frac 1 \tau \Ind_{\sum_{s\le \tau} (\Delta \zeta_s)^\theta > x}\right]\\
 \label{eq:dion}
 &= C\E\left[\frac 1 \tau \Ind_{\sum_{s\le \tau} (\Delta \zeta_s)^\theta > x,\,\tau > t_x}\right] + C\E\left[\frac 1 \tau \Ind_{\sum_{s\le \tau} (\Delta \zeta_s)^\theta > x,\,\tau \le t_x}\right].
 \end{align}
By \eqref{eq:tau_infty}, we bound the first summand on the right-hand side of \eqref{eq:dion} for large $x$ by
\begin{equation}
\label{eq:celine}
 \E\left[\frac 1 \tau \Ind_{\sum_{s\le \tau} (\Delta \zeta_s)^\theta > x,\,\tau > t_x}\right] \le \E\left[\frac 1 \tau \Ind_{\tau > t_x}\right] \le C(t_x)^{-(1/\alpha + 1)}.
\end{equation}
As for the second summand, we use H\"older's inequality to get for every $\ep>0$,
\begin{equation}
\label{eq:dicaprio}
 \E\left[\frac 1 \tau \Ind_{\sum_{s\le \tau} (\Delta \zeta_s)^\theta > x,\,\tau \le t_x}\right] \le \E\left[\frac 1 {\tau^{1/\ep}}\right]^{\ep} \P\left(\sum_{s\le \tau} (\Delta \zeta_s)^\theta > x,\,\tau \le t_x\right)^{1-\ep}.
\end{equation}
By \eqref{eq:tau_zero},
\begin{equation}
\label{eq:lemoine}
 \E\left[\frac 1 {\tau^{1/\ep}}\right]^{\ep} \le C_\ep.
\end{equation}
We continue with bounding the probability on the right-hand side in \eqref{eq:dicaprio}. Since $\zeta$ is $\alpha$-stable, there is a constant $\rho\in(0,1)$, such that $\P(\zeta_s \le 0) = \rho$ for all $s\ge 0$. Applying this with the Markov property at time $\tau$, we get
\begin{align}
\nonumber
&\rho \P\left(\sum_{s\le \tau} (\Delta \zeta_s)^\theta > x,\,\tau \le t_x\right)\\
\nonumber
 &= \P\left(\sum_{s\le t_x} (\Delta \zeta_s)^\theta > x,\,\tau \le t_x,\,\zeta_{t_x} \le -1\right)\\
\nonumber
 &\le \P\left(\sum_{s\le t_x} (\Delta \zeta_s)^\theta > x,\,\max_{s\le\tau} (\Delta \zeta_s)^\theta \le \delta x\right) + \P\left(\max_{s\le t_x} (\Delta \zeta_s)^\theta > \delta x,\,\zeta_{t_x} \le -1\right),
\end{align}
where $\delta$ is some small positive constant.
Classical large-deviation estimates for sums of iid heavy-tailed random variables \cite{Nagaev1979} yield that the first term on the right-hand side is for large $x$ smaller than any fixed polynomial in $x$ as long as $t_x \le x^{\alpha/\theta - \ep}$ and $\delta$ is sufficiently small. As for the second term on the right-hand side, denote by $\widetilde\zeta$ a process defined as $\zeta$, except that all the jumps greater than $(\delta x)^{1/\theta}$ are suppressed. Then, by independence of the large and small jumps,
\[
 \P\left(\max_{s\le t_x} (\Delta \zeta_s)^\theta > \delta x,\,\zeta_{t_x} \le -1\right) \le \P(-\widetilde\zeta_{t_x} \ge (\delta x)^{1/\theta}).
\]
One easily checks that as $x\to\infty$, $|\E[\widetilde\zeta_{t_x}]| \sim t_x C (\delta x)^{-(\alpha-1)/\theta} = o(\delta x)^{1/\theta})$ as long as $t_x = o(x^{\alpha/\theta})$. Hence, the event in the above probability is a large deviation event. By the finiteness of the moment generating function of $-\widetilde \zeta$ for positive values of the argument, the Chernoff bound easily implies that the probability $\P(-\widetilde\zeta_{t_x} \ge (\delta x)^{1/\theta})$ decays streched exponentially in $x$, as long as  $t_x \le x^{\alpha/\theta - \ep}$.

Summarizing the previous arguments, the probability
\[
 \P\left(\sum_{s\le \tau} (\Delta \zeta_s)^\theta > x,\,\tau \le t_x\right) 
\]
decreases superpolynomially in $x$ as long as  $t_x \le x^{\alpha/\theta - \ep}$. Together with \eqref{eq:dion}, \eqref{eq:celine}, \eqref{eq:dicaprio} and \eqref{eq:lemoine}, we have for every $\ep > 0$, for large $x$,
\[
 \P(W_1(\alpha,\theta) > x) \le C_\ep (t_x)^{-(1/\alpha + 1)},
\]
as long as $t_x \le x^{\alpha/\theta - \ep}$. Choosing $t_x$ to be equal to this bound, this gives for every $\ep>0$, for large $x$,
\(
\P(W_1(\alpha,\theta) > x) \le C_\ep x^{-(1+\alpha)/\theta +\ep}.
\)
This readily implies the statement of the lemma.
\end{proof}

\section{Convergence towards the continuous multiplicative cascade}
\label{sec:cv L infty}
In this section we prove our main result Theorem \ref{th:main}. We do it step by step in order to emphasize the different requirements for the different types of convergence.
\subsection{Finite dimensional convergence}

\begin{proposition}[Finite dimensional convergence] \label{prop:fidi} With the notation of Theorem \ref{th:main} we have the following convergence in distribution in the sense of finite-dimensional marginals 
$$ ( p^{-1}\chi\0p(u) : u \in \U) \quad \xrightarrow[p\to\infty]{(d)} \quad  ( Z_\alpha(u) : u \in \U).$$
\end{proposition}
\begin{proof}
This is a more or less straightforward corollary of the convergence of the first generation (Proposition \ref{prop:1ere gen}) together with the Markov property in the gasket decomposition. Recall the notation of the introduction and in particular $\chi\0p(\varnothing) = p$, recall also that $(\xi_i\0u\! : i \ge 1)_{u \in \U}$ are independent random variables distributed according to $\nu_\alpha$ and indexed by $ \U$.  Fix $k_0 \geq 1$. It follows from Proposition \ref{prop:1ere gen} that we have 
$$ \m({\frac{\chi\0p(i)}{\chi\0p(\varnothing)}}_{1 \leq i \leq k_0} \xrightarrow[p\to\infty]{(d)} \m({\xi_i^{(u)}}_{1 \leq i \leq k_0}.$$ Now, it follows from the gasket decomposition that conditionally on the perimeters $( \chi\0p(i) : 1 \leq i \leq k_0)$ of the first generation of the loops, the loop-decorated quadrangulations filling in the first $k_0$ holes (ranked in decreasing order of their perimeters) in the gasket are independent and distributed according to $ \prob_{(n;g,h)}^{\chi\0p(i)}$. Since we have $\chi\0p(k_0) \to \infty $ in probability (indeed $  \frac{1}{p}\chi\0p(k_0) \to Z_\alpha(k_0)$ in distribution and $Z_\alpha(k_0) >0$ almost surely) we can then apply Proposition \ref{prop:1ere gen} once more to these second generation quadrangulations to deduce that
$$ \m({\frac{\chi\0p( ij)}{\chi\0p(i)}}_{1 \leq j \leq k_0} \xrightarrow[p\to\infty]{(d)} \m({\xi^{(i)}_{j}}_{ 1 \leq j \leq k_0},$$
and these convergences in law hold jointly for all $i \in \varnothing \cup \{1,2, ... , k_0\}$. Iterating the above argument we get that for any finite subtree $ \mathfrak{t} \subset \U$ containing the root and for any vertex $u,ui \in  \mathfrak{t}$ we have the joint convergences $ \chi\0p(ui)/\chi\0p(u) \to \xi^{(u)}_i$ in distribution. This implies the finite dimensional convergence. 
\end{proof}

\subsection{$\ell^\infty$ convergence generation by generation}
\newcommand*{\sumu}[1]{\sum_{u\in\U:\abs{u}#1}\!\!}
In this subsection we strengthen Proposition \ref{prop:fidi} into a convergence in $\ell^{\infty}$ for any finite number of generations. Indeed, the convergence of Proposition \ref{prop:fidi} does not prevent $\chi\0p$ from having $\chi\0p(u) \asymp p$ at some vertex $u \to \infty$ (i.e.\ $u$ leaves any fixed finite subset of $\U$) as $p \to \infty$.
The following statement shows that this is impossible if the height of $u$ stays bounded.
For any $k \geq 1$ let $\U_k$ be subtree of the first $k$ generations in $ \U$.

\begin{proposition}\label{prop:convergence Uk}
For any $k \geq 1$ we have the following convergence in distribution in $\ell^\theta(\U_k)$ for all $\theta > \alpha$
$$ (p^{-1}\chi\0p(u) : u \in \U_k) \quad \xrightarrow[p\to\infty]{(d)} \quad ( Z_\alpha(u) : u \in \U_k).$$
\end{proposition}
As it will turn out, the last proposition is a consequence of the finite-dimensional convergence (Proposition \ref{prop:fidi}) together with a convergence in mean of the sum of powers in the cascade. More precisely we will use:
\begin{lemma}\label{lem:cv Lp by generation}
For any $\theta \in (\alpha, \alpha+1)$ and for any $k \geq 1$ we have 
$$  \exptm{ \sumu{=k} \m({ p^{-1} \chi\0p(u)}^\theta} \quad \cvg{}{p\to\infty} \quad \exptm{\sumu{=k} \mb({Z_\alpha(u)}^\theta} = \phi_\alpha(\theta)^k.$$
\end{lemma}

\begin{proof}
We prove the convergence by induction on $k\geq 1$. Using the notation of Section \ref{sec:randomwalkcoding} and in particular  \eqref{eq:apresluk} as well as \eqref{eq:1st gen}  we have 
$$  \mathbb{E}\left[\sum_{|u|=k} (p^{-1} \chi^{(p)}(u))^{\theta}\right] \leq  \frac{1}{ \mathbb{E}[1/L_{p}]} \mathbb{E}\left[ \frac{1}{L_{p}}\sum_{1\leq i \leq T_{p}} \left(\frac{X_{i}+1}{p}\right)^{\theta}\right] + O(p^{-1}).$$
The inequality comes from the fact that some faces (of degree four) of the gasket are not holes surrounded by a loop and the $O(p^{-1})$ from the approximation $L_p \approx L_p+1$ (Recall that $p\le L_p\le T_p$). By Fatou's lemma the right-hand side in Lemma \ref{lem:cv Lp by generation} is less than the liminf of the left-hand side. Therefore it suffices to prove that the right-hand side in the last display converges towards $\phi_{\alpha}(\theta)$. This will be shown using Theorem \ref{th:rw} and the approximation
\begin{equation}\label{eq:approx}
\mujs(0)\frac1{L_p}\sum_{i=1}^{T_p} \m({\frac{X_i+1}p}^\theta\ \approx \,\
U\0\theta_p := \frac1{T_{p}}\sum_{i=1}^{T_p} \m({\frac{X_i+1}p}^\theta.
\end{equation}
Indeed since $T_p \geq p$, by Theorem \ref{th:rw} we can write
$$ \mathbb{E}[U_{p}^{(\theta )}] = \mathbb{E}\left[ \frac{1}{T_{p}-1} \sum_{i=1}^{T_{p}} \left(	 \frac{X_{i}+1}{p} \right)^{\theta}\right]\cdot (1+O(p^{-1})) =  \exptm{ \mB({\frac{X_1+1}{p}}^\theta \frac{p}{p+X_1} }\cdot (1+O(p^{-1})).$$
Recall that $ p^\alpha\,\prob(X_1\ge xp) \to \pi([x,\infty))$ where $\pi(\dd x) = C x^{-\alpha-1} /\Gamma(-\alpha)\idd{x>0}$ is the L\'evy measure of the $\alpha$-stable L\'evy process $\zeta$. As in the proof of Proposition \ref{prop:levy}, we can use dominate convergence to get that 
$$ p^\alpha\cdot \expt[U\0\theta_p]\ \xrightarrow[p\to\infty]{} \
\int \frac{x^\theta}{1+x} \pi(\dd x)\ =\ 
\expt\mh[{\frac1\tau\sum_{t\le\tau} (\Delta\zeta_t)^\theta}.$$
We have already seen in the proof of Theorem \ref{prop:1ere gen} that 
$p^\alpha \cdot \expt[\mujs(0)/L_p] \to \expt[1/\tau]$.
Provided that the approximation \eqref{eq:approx} holds in expectation with an error of order $o(p^{-\alpha})$, we can gather the pieces and indeed deduce the desired convergence for $k=1$ by the calculation done in Section \ref{sec:calculstable}. 

To justify \eqref{eq:approx}, for our fixed $\theta\in(\alpha,\alpha+1)$, let $\gamma,\gamma'>1$ be such that $\gamma\theta<\alpha+1$ and $\frac1\gamma+\frac1{\gamma'} =1$, then
\begin{align*}
	\abs{ \expt\mh[{ \mujs(0)\cdot\frac{T_p}{L_p} U_p\0\theta }
		- \expt\mB[{U_p\0\theta}	}
&\le \exptm{ \mB|{\mujs(0)\frac{T_p}{L_p}-1} \cdot U_p\0\theta}
&&\\	&\le 
	\exptm{ \mB|{ \mujs(0)\frac{T_p}{L_p} -1 }^{\gamma'}
		  }^{ \frac{1}{\gamma'} }
\!\!\!\cdot \exptm{ (U_p\0\theta)^\gamma }^{ \frac1\gamma }
&&\text{(H\"older's inequality)}\\	&\le 
	\exptm{ \mB|{ \mujs(0)\frac{T_p}{L_p} -1 }^{\gamma'}
		  }^{ \frac1{\gamma'} }
\!\!\!\cdot	\exptm{ U_p\0{\gamma\theta} }^{ \frac1\gamma }
&&\text{(Jensen's inequality)}
\end{align*}
The large deviation bound \eqref{eq:bound 2} and \eqref{eq:bound 3} imply that $\exptm{ \abs{\mujs(0)T_p/L_p -1}^{\gamma'}}$ tends to $0$ faster than $p^{-\alpha \gamma'/4}$ as $p \to \infty$. On the other hand, the above calculation shows that $ \mathbb{E}[U_{p}^{(\gamma \theta)}]$ is of order $p^{-\alpha}$. Hence the right-hand side of the last display is of order $o(p^{-\alpha/4 - \alpha/\gamma})$. By choosing $\gamma$ close enough to 1 this can be made smaller than $p^{-\alpha}$. This justifies our approximation \eqref{eq:approx}.

Now assume that the convergence of the lemma takes place up to the $k$-th generation and write $$
G_{k}(p) = \mathbb{E}\left[ \sum_{|u|=k} \left( p^{-1} \chi^{(p)}(u)\right)^{\theta} \right],$$ to simplify notation. 
Then there exists a constant $C = C(k,\theta)$ such that $G_{k}(p) \leq C$ for all $p \geq 1$, and for any $\varepsilon>0$ there exists $p_0$ such that for all $p\ge p_0$,
$$	G_{k}(p)
\le \phi_\alpha(\theta)^k+\varepsilon.	$$
Using the Markov property of the gasket decomposition at the first generation we get with the above two inequalities, for all $p\ge 1$,
\begin{align*}
	 G_{k+1}(p)
&=	 \exptm{
			\sum_{i=1}^\infty
		 	\m({  \frac{\chi\0p(i)}p }^\theta G_{k}{(\chi^{(p)}(i))}}
\\&\le	\m({		\phi_\alpha(\theta)^k + \varepsilon		}
		\exptm{
			\sum_{i=1}^\infty
			\m({ \frac{\chi\0p(i)}p }^\theta \idd{\chi\0p(i)\ge p_0}	 	
			} 
		+ C \exptm{
			\sum_{i=1}^\infty 
			\m({ \frac{\chi\0p(i)}p }^\theta \idd{\chi\0p(i)< p_0}}
\end{align*}
By the $k=1$ case, the first term on the right-hand side is bounded using 
\[
 \exptm{
			\sum_{i=1}^\infty
			\m({ \frac{\chi\0p(i)}p }^\theta\idd{\chi\0p(i)\ge p_0}	 	 	
			}
			\le
			 \exptm{
			\sum_{i=1}^\infty
			\m({ \frac{\chi\0p(i)}p }^\theta	 	
			}
			\quad \cvg{}{p\to\infty} \quad \phi_\alpha(\theta).
\]
As for the second term, fix $\theta'\in(\alpha,\theta)$ then for $p \geq p_{0}$ we can write
\[
 \exptm{
			\sum_{i=1}^\infty 
			\m({ 
			\frac{\chi\0p(i)}p }^\theta \idd{\chi\0p(i)< p_0}
			}
			\le 
			\m({\frac {p_0}p}^{\theta-\theta'}
			\exptm{
			\sum_{i=1}^\infty 
			\m({ \frac{\chi\0p(i)}p }^{\theta'}
			}
			\quad \cvg{}{p\to\infty} \quad 0,
\]
by the $k=1$ case proven above (with $\theta$ replaced by $\theta'$). Taking the limits $p\to\infty$ and then $\varepsilon\to0$, we get the upper bound
$$	\limsup_{p\to\infty} G_{k+1}(p)
\,\le\,	\phi_\alpha(\theta)^{k+1},	$$
whereas the lower bound $\liminf_{p\to\infty} G_{k+1}(p)
\,\ge\,	\phi_\alpha(\theta)^{k+1}$ is trivial from the finite-dimensional convergence together with Fatou's lemma.
\end{proof}
 
\begin{proof}[Proof of Proposition \ref{prop:convergence Uk}]
Since the identity function $\iota: \ell^\theta(\U_k)\to\ell^{\theta'}(\U_k)$ is continuous for all $\theta\le\theta'$, it suffices to prove the convergence in distribution for all $\theta$ close enough to $\alpha$.
Fix $\theta\in(\alpha,\alpha+1)$ and $k \geq 1$. Since $\exptm{\sum_{u\in \U_k}(Z_\alpha(u))^\theta} = 1 +\phi_\alpha(\theta) + \dots + \phi_\alpha(\theta)^k < \infty$, for any $\varepsilon>0$ there is a finite subset $V\subset \U_k$ such that
$$\exptm{ \sum_{u\in\U_k\setminus V} \m({Z_\alpha(u)}^\theta } <\varepsilon^{\theta}.$$
According to the convergence in Lemma \ref{lem:cv Lp by generation}, we have
\begin{equation}\label{eq:fidi to Uk}
	\limsup_{p \to \infty} \exptm{ \sum_{u\in\U_k\setminus V} \m({p^{-1}\chi\0p(u)}^\theta } \le \varepsilon^{\theta}.
\end{equation}
Now if $f : \ell^{\theta}( \mathcal{U}_{k}) \to \mathbb{R}_{+}$ is a bounded $K$-Lipschitz function we have 
$$ \mB|{ \exptm{f(Z_\alpha)} - \exptm{f(Z_\alpha\id_V)}} 
\ \le\ K \cdot \exptm{ \mB|{Z_\alpha- Z_\alpha\id_V}_\theta} 
\ \underset{ \text{H\"older}}\le\ K \cdot  \m({\exptm{ \sum_{u\in\U_k\setminus V} (Z_\alpha(u))^\theta }}^{1/\theta}\!\!\! \le\, K \cdot \varepsilon \,,$$ 
and similarly $\expt[f( p^{-1} \chi\0p)] \approx  \expt[f(p^{-1}\chi\0p\id_V)]$ up to an error of order $ \varepsilon$ as $p\to\infty$. From the finite-dimensional convergence (Proposition \ref{prop:fidi}) we deduce that $\exptm{f(p^{-1} \chi\0p \id_V)} \to \exptm{f(Z_\alpha \id_V)}$ as $p\to\infty$. Put all together this shows $ \exptm{f( p^{-1} \chi\0p)} \to \exptm{f(Z_\alpha)}$ as $p\to\infty$ and proves the desired convergence in distribution. \end{proof}

\subsection{$\ell^\infty$ convergence}
As we already notice, Proposition \ref{prop:convergence Uk} implies the convergence of $p^{-1}\chi\0p\longrightarrow Z_\alpha$ in $ \ell^{\infty} ( \mathcal{U}_{k})$. However, it does not yet yield the full convergence in $ \ell^{\infty}( \mathcal{U})$ and the missing estimate is of the form: for any $\varepsilon>0$, there exists an integer $k$ such that
\begin{equation}\label{eq:sous-marin proba bound}
\limsup_{p\to\infty} \proba\m({ \exists\, u\in \U\setminus\U_k \,:\, \chi\0p(u)\ge \varepsilon p} \le \varepsilon.
\end{equation}
In other words, the labels beyond generation $k$ are uniformly small when $k$ is large. Notice that if we had replaced $p^{-1}\chi\0p$ by the limiting cascade $Z_\alpha$, the estimate would be immediate: by the remark  after \eqref{avantrek}, the process $(Z_{\alpha})$ almost surely belongs to $\ell^{ \theta}( \mathcal{U})$ for a certain $\theta >0$. Our way to prove \eqref{eq:sous-marin proba bound} is similar as in the continuous case and we want to find a supermartingale of the form 
$$\left(\sum_{|u|=k} f\mb({\chi\0p(u)}\right)_{k\ge 0}$$ where $f$ is an increasing function.
The underlying quadrangulation model provides us naturally one such supermartingale:
for $p\ge 1$, let $\bV(p)$ be the expected volume (i.e.\ number of vertices) of a random loop-decorated quadrangulation of distribution $\prob\0p_{n;g,h}$.
Then the gasket decomposition (Section \ref{sec:gasket}) immediately shows that we have the strict inequality for all $p$:
\begin{equation}\label{eq:volume <}
 \exptm{ \sum_{i=1}^\infty \bV\m({\chi\0p(i)} } <\bV(p).
\end{equation}
In particular, $(\sum_{|u|=k} \bV(\chi\0p(u)),\,k\ge 0)$ is indeed a supermartingale for the discrete cascade.
Timothy Budd recently proved the following asymptotics of $\bV(p)$:
\begin{citetheorem}[\cite{BudOn}]\label{th:budd}
For each set of non-generic critical parameters $(n;g,h)$, we have $\bV(p) \underset{p\to\infty}\sim \Lambda p^{\theta_\alpha}$ where $\theta_\alpha = \max(2\alpha-1,2)$ and $\Lambda>0$ is some constant depending on $(n;g,h)$.
\end{citetheorem}
With this estimate, one can proceed to the proof of Theorem \ref{th:main}.

\begin{proof}[Proof of the $\ell^\infty(\U)$ convergence]
Recall that we assume $n<2$, so that $\alpha\ne \frac{3}{2}$ and $\inf\phi_\alpha <1$.
Choose $\theta$ such that $\phi_\alpha(\theta)<1$.
Then by Lemma \ref{lem:cv Lp by generation}, there exist finite constants $C$, $p_0$ and $c<1$ such that
\begin{equation}\label{eq:sous-marin induction init}
\exptm{\sum_{i=1}^\infty \mB({ p^{-1} \chi\0p(i) }^\theta } \le 
\begin{cases}	C	&\text{for all }p\ge 1
			\\	c	&\text{for all }p\ge p_0
\end{cases}
\end{equation}
This inequality indicates that the $\theta$-moment decreases exponentially as long as the labels do not drop below $p_0$ too often.
To make this idea precise, for $u \in \mathcal{U}$ let $N_0(u)$ be the number of ancestors of $u$ which have a label smaller than $p_0$.
The following lemma controls the size of $\chi\0p(u)$ depending on whether $N_0(u)$ is smaller or greater than a threshold $m$.
\begin{lemma}\label{lem:sous-marin 2 bounds}
\begin{enumerate}
\item	For all $k\ge m\ge 0$ we have
	\begin{equation}\label{eq:sous-marin moment bound}
	\exptm{ \sumu{=k} \m({ \frac{\chi\0p(u)}{p} }^\theta \idd{N_0(u)\le m}} 
	\le C^m c^{k-m}.
	\end{equation}
\item	Consider the set of vertices 
	$L_{m} = \inf\{u\in\U: N_0(u)\geq m\}$
	where the infimum of a subset $U\subset\U$ is defined by 
	$\inf U = \{u\in U: u\text{ has no ancestor in }U\}$.
	Then there exists $\tilde{c}<1$ such that for all $p$ and $m$,
	\begin{equation}\label{eq:sous-marin volume bound}
		\exptm{\sum_{u\in L_{m}} \bV\mb({\chi\0p(u)} }
		\le \tilde{c}^{m-1}\bV(p)
	\end{equation}
\end{enumerate}
\end{lemma}
\noindent By Markov's inequality, \eqref{eq:sous-marin moment bound} and \eqref{eq:sous-marin volume bound} imply respectively
\begin{align*}
&&	\prob\m({\exists\,u\in\U\setminus\U_k: N_0(u)\le m
							\text{ and } \chi\0p(u)\ge\varepsilon p}
&\le \varepsilon^{-\theta}\sum_{l>k} C^m c^{l-m}
=	\varepsilon^{-\theta} \mB({\frac{C}{c}}^m \frac{c^{k+1}}{1-c}
\\	&\text{and}&
	\prob\m({\exists\,u\in\U: N_0(p)>m
				\text{ and } \chi\0p(u)\ge\varepsilon p  }
&\le \tilde{c}^m \frac{\bV(p)}{\bV(\varepsilon p)}
\end{align*}
Take the sum of the two inequalities and use the regular variation of $\bV$ to show that
$$ \limsup_{p\to\infty} \proba\m({ \exists\, u\in \U\setminus\U_k \,:\, \chi\0p(u)\ge \varepsilon p} 
\le \varepsilon^{-\theta} \mB({\frac{C}{c}}^m \frac{c^{k+1}}{1-c}
	+ \tilde{c}^m \varepsilon^{-\theta_\alpha}.$$
The right-hand side tends to zero when $k,m\to\infty$ under the constraint $\frac{k}{m}\ge \frac{\log C - \log c}{-\log c} +\varepsilon$.
This proves the bound \eqref{eq:sous-marin proba bound} and the $\ell^\infty(\U)$ convergence in Theorem \ref{th:main} modulo Lemma \ref{lem:sous-marin 2 bounds}.
\end{proof}

\begin{proof}[Proof of Lemma \ref{lem:sous-marin 2 bounds}]
We prove the bound \eqref{eq:sous-marin moment bound} by induction on $k$. We write 
$$ M_{k,m}(p)=\exptm{ \sumu{=k} \m({ \frac{\chi\0p(u)}{p} }^\theta \idd{N_0(u)\le m}}$$  to simplify notation. In the case $k=1$ (and $m \in \{0,1\}$) the only ancestor of the first generation is the root and the estimate follows from \eqref{eq:sous-marin induction init}. If $k \geq 1$ then we distinguish according to:

\begin{itemize}
\item If $p\ge p_0$, then for all $i\ge 1$ and $u\in\U$, we have $N_0(iu) = N_0\0i(u)$, where $N_0\0i$ is a copy of the function $N_0$ defined on the sub-tree rooted at the vertex $i$.
It follows that
\begin{align*}
	M_{k+1,m}(p) 
&=	\exptm{ \sum_{i=1}^\infty \m({ \frac{\chi\0p(i)}p }^\theta
		 	 \sumu{=k} \m({ \frac{\chi\0p(iu)}{\chi\0p(i)} }^\theta 
		 	 \idd{N\0i_0(u)\le m}	}
\\&=\exptm{ \sum_{i=1}^\infty \m({ \frac{\chi\0p(i)}p }^\theta
			M_{k,m}\mb({\chi\0p(i)}	}
			\tag{Markov property of the cascade}
\\&\le	C^m c^{k-m}\ 
	\exptm{\sum_{i=1}^\infty \mB({ p^{-1} \chi\0p(i) }^\theta }
 			\tag{induction hypothesis}
\\& \le	C^m c^{k+1-m}.
\end{align*}	
\item If $p<p_0$, then we have $N_0(iu) = N\0i_0(u) + 1$ and hence for $m\ge 1$,
\begin{align*}
	M_{k+1,m}(p)
&= \exptm{ \sum_{i=1}^\infty \m({ \frac{\chi\0p(i)}p }^\theta
			M_{k,m-1}\mb({\chi\0p(i)}	}
\\&\le	C^{m-1} c^{k-(m-1)}\ 
	\exptm{\sum_{i=1}^\infty \mB({ p^{-1} \chi\0p(i) }^\theta } 
\ \le	C^m c^{k+1-m}.
\end{align*}
For $m=0$, we have $M_{k+1,0}(p)=0$ since $p<p_0$. This completes the induction.
\end{itemize}

Let us move to the second point of the lemma. To show \eqref{eq:sous-marin volume bound}, first remark that \eqref{eq:volume <} implies the existence of a constant $\tilde{c}<1$ such that
\begin{equation}\label{eq:volume <=}
\exptm{ \sum_{i=1}^\infty \bV\m({\chi\0p(i)} } \le \tilde{c}\, \bV(p)
\end{equation}
for all $p\le p_0$.
To simplify notation, we will write
$\hat{V}(U) = \exptm{\sum_{u\in U} \bV(\chi\0p(u))}$ for any subset $U\subset \U$.

For $k\ge 1$, let $L_k = \inf\{ u\in\U : N_0(u)\ge k \}$ and $L_k^+ = \{ui : u\in L_k,\, i\in\natural^*\}$ ($L_k^+$ is the set of children of $L_k$).
From the definition of $N_0(.)$ it is not hard to see that $\chi\0p(u) \le p_0$ for all $u\in L_k$. The random sets $L_k$, $L_k^+$ are  so-called \emph{optional lines} for the filtration generated by the process $\chi\0p$ (see e.g.~\cite{BK04}) and we have 
$$ \{\varnothing\}=L_{0}\prec L_{0}^{+}\preceq L_{1} \prec L_{1}^{+} \preceq L_{2} \prec L_{2}^{+} \preceq L_{3} \prec \cdots$$
where we used the partial order on the subsets of $\U$ defined by $U \preceq \tilde{U}$ if each vertex $u\in\tilde{U}$ either is in $U$ or has an ancestor in $U$. On the one hand, by general theory on optional lines, if $L\preceq L'$ are two optional lines then $ \hat{V}(L)\ge \hat{V}(L') $. On the other hand, since $\chi\0p(u) \le p_0$ for all $u\in L_k$ we can use \eqref{eq:volume <=} to deduce that 
$$ 	\hat{V}(L_k^+) 
=	\exptm{ \sum_{u\in L_k} \bV\m({ \chi\0p(u) } \cdot
			\exptm{ \sum_{i=1}^\infty \frac{\bV\m({\chi\0q(i)}}{\bV(q)}
			}_{q=\chi\0p(u)} }
\le \tilde{c}\, \hat{V}(L_k).	$$
Gathering the two inequalities we indeed deduce that $L_{m+1} \leq \tilde{c}^m\hat{V}(L_{0}) = \tilde{c}^m \bV(p),$ as desired.\end{proof}

\appendix
\section{Relation with other nesting statistics}
\label{sec:KPZ}

In this section, we outline the relation of our work to the recent work by Borot, Bouttier and Duplantier \cite{BBD16} about the number of loops surrounding a typical vertex in a $O(n)$-decorated random planar map and to analogous quantities in conformal loop ensembles. 

\paragraph{Number of loops surrounding a typical vertex in the $O(n)$-decorated quadrangulation.}
We consider a random \emph{pointed} quadrangulation of (large) perimeter $p$ decorated with an $O(n)$ loop model, as defined in the introduction of the main text. Borot, Bouttier and Duplantier \cite{BBD16} have studied the large deviations of the number of loops surrounding the marked vertex, by methods from analytic combinatorics. With our notation, their result reads as follows:
\begin{citetheorem}[\cite{BBD16}, Theorem~2.2]
\label{th:bbd}
Let $N$ denote the number of loops surrounding the marked vertex. Then, for all $x>0$, as $p\to\infty$,
\begin{align*}
&\frac{1}{\log p}\log \P(N = \lfloor x \log p \rfloor) \longrightarrow -\frac 1 \pi J(\pi x),\\
 \text{where}\quad &J(x) = x \log\left(\frac 2 n \frac{x}{\sqrt{1+ x^2}}\right) + \arccot(x) - \arccos(\tfrac n 2).
\end{align*}
\end{citetheorem}

In fact, the result in \cite{BBD16} is more precise in that the authors actually give an asymptotic equivalent for $\P(N= \lfloor x \log p \rfloor)$. Also note that there is a mistake in the definition of the function $J$ in \cite{BBD16} (the first $x$ factor is missing). 

We now sketch how we can heuristically recover this result from the continuous multiplicative cascade $Z_{\alpha}$ of Theorem \ref{th:main}. 
Let $\delta>0$ be a small constant. We define $\mathcal L_\delta$ to be the set of those vertices $u$ in the Ulam tree for which $Z_\alpha(u) < \delta$ and $Z_\alpha(v) \ge \delta$ for every ancestor $v$ of $u$. ($\mathcal{L}_\delta$ is an optional line, see the proof of Lemma~\ref{lem:sous-marin 2 bounds}.) In the discrete setting, these vertices correspond to loops in a  $O(n)$-decorated quadrangulation of perimeter $p$ whose perimeter is smaller than $\delta p$, but the loops surrounding them have perimeter larger than $\delta p$. 

Similarly to the definition of the martingale $W_n(\alpha,\theta)$ in \eqref{eq:martingale} we now define
\[
 W^\delta(\alpha,\theta) = \sum_{u\in\mathcal L_\delta} Z_\alpha(u)^\theta\phi_\alpha(\theta)^{-|u|},
\]
where as usual, $|u|$ denotes the generation of a vertex $u$ in the Ulam tree.
One can then show (for example with the methods from \cite{BK04}) that $\E[W^\delta(\alpha,\theta)] = 1$ for every $\theta \in(\alpha, \operatorname{argmin}_\theta \phi_\alpha(\theta)) = (\alpha,\alpha + 1/2)$.

As a consequence, we have for such $\theta$,
\begin{equation*}
 1 = \E[W^\delta(\alpha,\theta)] = \E\left[\sum_{u\in\mathcal L}  Z_\alpha(u)^\theta\phi_\alpha(\theta)^{-|u|}\right] \approx \delta^{\theta} \E\left[\sum_{u\in\mathcal L_\delta} \phi_\alpha(\theta)^{-|u|}\right].
\end{equation*}
Now, as said before, every $u\in\mathcal L_\delta$ roughly corresponds to a loop in the $O(n)$ model of perimeter $\delta p$ and $|u|$ is then the number of loops surrounding it. Assuming we could take $ \delta = \frac{1}{p}$ this suggests that 
\begin{equation}
\label{eq:on}
 \E[\sum_v \phi_\alpha(\theta)^{-N(v)}] \approx p^{\theta},\quad \theta\in(\alpha,\alpha+1/2),
\end{equation}
where the sum is on the vertices of the loop-decorated quadrangulation and $N(v)$ is the number of loops separating the vertex $v$ from the outerface. 

We now write \eqref{eq:on} in a different form in order to link it to Theorem~\ref{th:bbd}. First recall from Section~\ref{sec:volume} (or Theorem~\ref{th:budd}) that the volume of the $O(n)$-decorated quandrangulation scales as $p^{\theta_0}$, where $\theta_0 = \min(2,2\alpha-1)$. Equation~\eqref{eq:on} is then equivalent to
\[
\E[\phi_\alpha(\theta)^{-N}] \approx p^{\theta - \theta_0},\quad \theta\in(\alpha,\alpha+1/2),
\]
where $N$, as in Theorem~\ref{th:bbd}, is now the number of loops surrounding the marked vertex in a \emph{pointed} $O(n)$-decorated quadrangulation. This allows to express the moment generating function of $N$ by
\begin{equation}
\E[e^{\lambda N}] \approx p^{\kappa_\alpha(\lambda)},
\label{eq:kappa0}
\end{equation}
where $\kappa_\alpha(\lambda) = \phi_\alpha^{-1}(e^{-\lambda}) - \theta_0$, with $\phi_\alpha^{-1}$ the inverse of the restriction of $\phi_\alpha$ to $(\alpha,\alpha+1/2)$ and $\lambda < -\log \min_\theta \phi_\alpha(\theta)$. Now, \eqref{eq:alpha_n} gives $\sin(\pi(2-\alpha)) = n/2$ and $\theta_0 - \alpha = \pi^{-1}\arcsin(n/2)$, so that we can express $\kappa_\alpha$ by
\begin{equation}
\kappa_\alpha(\lambda) = \frac{1}{\pi}\mB({\arcsin(\tfrac n 2 e^\lambda) - \arcsin(\tfrac n 2)} = \frac{1}{\pi}\mB({\arccos(\tfrac n 2) - \arccos(\tfrac n 2 e^\lambda)},\quad \lambda < \log(\tfrac 2 n),
\label{eq:kappa}
\end{equation} 
and $\kappa_\alpha(\lambda) = +\infty$ for $\lambda > \log(\tfrac 2 n)$, by convexity.

Equation~\eqref{eq:kappa0} now suggests that for every $x>0$, as $p\to\infty$,
\[
\frac 1 {\log p} \log \P(N = \lfloor x \log p\rfloor) \longrightarrow -\kappa_\alpha^*(x), 
\]
where $\kappa_\alpha^*(x) = \sup_{\lambda\in\R} \{\lambda x - \kappa_\alpha(x)\}$ is the Legendre--Fenchel transform of the function $\kappa_\alpha$. Using the explicit expression in \eqref{eq:kappa}, a simple calculation shows:
\[
\kappa_\alpha^*(x) = \frac 1 \pi J(\pi x),\quad x>0,
\]
where $J$ is the function from Theorem~\ref{th:bbd}.
This establishes (again, heuristically) that theorem.

\paragraph{Number of loops in a conformal loop ensemble surrounding a small Euclidean ball.}
We now show how one can heuristically relate \eqref{eq:on} to a similar statement for the number of loops in a conformal loop ensemble in the unit disk surrounding a small Euclidean ball, thereby recovering (again heuristically) a result by Miller, Watson and Wilson \cite{MillerWatsonWilson}. The argument is similar to the one by Borot, Bouttier and Duplantier \cite{BBD16}, but may be easier to understand since we avoid here the use of Legendre--Fenchel transforms. Recall from the introduction that it is conjectured that in a $O(n)$-decorated quadrangulation with a boundary, the volume measure together with the loops converges in some sense to the so-called Liouville quantum disk (with parameter $\gamma = \sqrt{\kappa}$) together with an independent $\mathrm{CLE}_\kappa$ in the disk, where $\kappa$ is related to our parameter $\alpha$ by $\alpha - \frac 3 2 = 4/\kappa - 1$ (see \eqref{eq:alpha_kappa}). For simplicity, we restrict ourselves to the dilute case ($\alpha > 3/2$, or $8/3 < \kappa < 4$). The result from \cite{MillerWatsonWilson} is the following: Let $\widetilde N_r$ denote the number of $\mathrm{CLE}_\kappa$ loops surrounding a fixed Euclidean ball of radius $r \ll 1$. Then \cite{MillerWatsonWilson},
\begin{equation}
\label{eq:fernanda_martins}
 \E[\psi_\kappa(\theta)^{-\widetilde N_r}] \approx r^{-\theta},
\end{equation}
where 
 \begin{equation}
  \label{eq:ssw}
  \psi_\kappa(\theta) = \frac{-\cos(4\pi/\kappa)}{\cos(\pi\sqrt{(1-4/\kappa)^2-8\theta/\kappa})} = \frac{\cos(\pi(\alpha-3/2))}{\cos(\pi \sqrt{(\alpha-3/2)^2-\theta(2\alpha-1)})}.
 \end{equation}
This function appeared already in \cite{SSW}. Note that we can express it as
 \begin{equation}
 \label{eq:apokalypsa}
  \psi_\kappa(\theta) = \phi_\alpha\left(1 + \frac 4 \kappa - \sqrt{(1-4/\kappa)^2 - 8\theta/\kappa}\right),
 \end{equation}
 where $\phi_\alpha$ is the Biggins transform of the multiplicative cascade $Z_\alpha$.

Here is an explanation for the relation \eqref{eq:apokalypsa}: 
Denote by $\mu_\kappa$ the Liouville quantum gravity measure in the disk. We can then discretize the disk into blocks of $\mu_\kappa$-mass approximately $\delta^2$, for example by a dyadic decomposition as in \cite{DS}. Such a block is then the analogue of a vertex of the $O(n)$-decorated quadrangulation of perimeter $p$, with $\delta = 1/p$ (recall that in the dilute phase, the volume scales like the perimeter squared). 

For each $c>0$, denote by $\nu_{\delta,\delta^{1/c}}$ the number of blocks of diameter approximately $\delta^{1/c}$. It is implicit in \cite{DS} that
\begin{equation}
\label{eq:sheefit}
 \nu_{\delta,\delta^{1/c}} \approx \delta^{-2 + \frac 2 {c\kappa}(c-(1-\kappa/4))^2}.
\end{equation}
For each block $b$, denote by $\widetilde N(b)$ the number of CLE loops surrounding the block. 
Equations \eqref{eq:fernanda_martins} and \eqref{eq:sheefit} suggest that
\begin{align*}
 \E[\sum_b \psi_\kappa(\theta)^{-\widetilde N(b)}] 
 &\approx \sup_{c>0} \nu_{\delta,\delta^{1/2c}} \times \E[\psi_\kappa(\theta)^{-\widetilde N_{\delta^{1/c}}}]\\
 &\approx  \sup_{c>0} \delta^{-2 + \frac 2 {c\kappa}(c-(1-\kappa/4))^2 - \frac{\theta}{c}}.
\end{align*}
A simple calculation shows that
\[
\inf_{c>0} \frac 2 {c\kappa}(c-(1-\kappa/4))^2 - \frac{\theta}{c} = 1 -\frac 4 \kappa + \sqrt{(1-4/\kappa)^2 - 8\theta/\kappa},
\]
for $\theta$ small enough. This gives
\begin{equation}
 \label{eq:off}
 \E[\sum_b \psi_\kappa(\theta)^{-\widetilde N(b)}] \approx \delta^{-1-\frac 4 \kappa + \sqrt{(1-4/\kappa)^2 - 8\theta/\kappa}}.
\end{equation}
On the other hand, by \eqref{eq:on} we expect that
\begin{equation}
 \label{eq:on_off}
 \E[\sum_b \phi_\kappa(\widetilde \theta)^{-\widetilde N(b)}] \approx \delta^{-\widetilde \theta},
\end{equation}
for $\widetilde \theta \in(\alpha,\alpha+1/2)$. Comparing \eqref{eq:off} and \eqref{eq:on_off} suggests that $\psi_\kappa(\theta) = \phi_\alpha(\widetilde \theta)$ if $\theta$ and $\widetilde\theta$ are related through 
$$\widetilde \theta = 1+\frac 4 \kappa - \sqrt{(1-4/\kappa)^2 - 8\theta/\kappa}.$$ This readily implies \eqref{eq:apokalypsa}.

\bibliographystyle{abbrv}
\bibliography{merged}
\Addresses

\end{document}